\documentclass[pagebackref=true,12pt]{amsart}
\usepackage{fullpage}
\usepackage{hyperref}
\hypersetup{
    colorlinks,
    citecolor=blue,
    filecolor=blue,
    linkcolor=blue,
    urlcolor=blue 
}
\usepackage{orcidlink}
\usepackage{bm}
\usepackage{fullpage}
\usepackage{graphicx}%
\usepackage{multirow}%
\usepackage{amsfonts}%

\usepackage{textcomp}%
\usepackage{manyfoot}%
\usepackage{CJKutf8}
\usepackage{booktabs}%
\usepackage{algorithm}%
\usepackage{algorithmic}%
\usepackage{listings}%
\usepackage{xcolor}
\usepackage{tikz}
\usepackage{tikz-cd}
\usepackage{tikz-3dplot}

\usepackage{amsmath,mathrsfs,amssymb,amsthm,amscd}
\usepackage{mathtools}
\usepackage{url}
\usepackage{natbib}
\usepackage{epigraph}
\usepackage[all]{xy}
\usepackage{wasysym}
\usepackage{footnote}

\makesavenoteenv{algorithm}
\newtheorem{theorem}{Theorem}[section]
\newtheorem{lemma}[theorem]{Lemma}

\theoremstyle{definition}
\newtheorem{definition}[theorem]{Definition}
\newtheorem*{question}{Question}

\theoremstyle{remark}
\newtheorem*{remark}{Remark}

\numberwithin{equation}{section}

\begin{document}

\title{Counting points on surfaces in polynomial time}

\author{Nitin Saxena \orcidlink{0000-0001-6931-898X}}
\address{Department of Computer Science \& Engineering, IIT Kanpur, India} 
\email{nitin@cse.iitk.ac.in}
\author{Madhavan Venkatesh}
\address{Department of Computer Science \& Engineering, IIT Kanpur, India} 
\email{madhavan@cse.iitk.ac.in}
 \thanks{}

\date{}
\vspace{1cm}
\dedicatory{To the memory of Sebastiaan Johan Edixhoven.}

\vspace{2cm}
\begin{abstract}

 We present a randomised algorithm to compute the local zeta function 
 of a fixed smooth, projective surface over $\mathbb{Q}$, at any large prime $p$ of good reduction. The runtime of our algorithm is polynomial in $\log p$, resolving a conjecture of Couveignes and Edixhoven. 
 

\end{abstract}

\maketitle
\vspace{1cm}
\setcounter{tocdepth}{1}
\tableofcontents

\newpage
\section{Introduction}
\subsection{Main result}
Let $X\subset \mathbb{P}^{N}$ be a fixed smooth, projective, geometrically integral (properties we abbreviate to \textit{nice}) surface of degree $D$ over a finite field $\mathbb{F}_{q}$, described by a system of homogeneous polynomial equations $f_{1}, \ldots, f_{m}$ each of degree $\leq d$. We assume $X$ is obtained via good reduction of a nice surface $\mathcal{X}$ over a number field $K$ at a prime $\mathfrak{p}\subset \mathcal{O}_{K}$. The zeta function of $X$ is
$$
Z(X/\mathbb{F}_{q}, T):=\exp\left(\sum_{j=1}^{\infty}\#X(\mathbb{F}_{q^{j}})\frac{T^{j}}{j}\right).
$$
Fix a prime $\ell$ coprime to $q$. From the Weil conjectures for $X$, we know that 
$$
Z(X/\mathbb{F}_{q}, T)=\frac{P_{1}(X/\mathbb{F}_{q}, T)P_{3}(X/\mathbb{F}_{q}, T)}{(1-T)P_{2}(X/ \mathbb{F}_{q}, T)(1-q^{2}T)},
$$
where $P_{i}(X/\mathbb{F}_{q}, T):=\det\left(1-TF_{q}^{\star} \ \vert \ \mathrm{H}^{i}(X, \mathbb{Q}_{\ell})\right)$ is the (reversed) characteristic polynomial of the geometric Frobenius acting on the $i^{\text{th}}$ \ $\ell $ -- adic \'etale cohomology group of $X$. In \cite[Epilogue]{ceramanujan}, the existence of an algorithm that computes the point count $\#X(\mathbb{F}_{q})$ in time polynomial in $\log q$ is conjectured. We prove this conjecture by exhibiting an algorithm that computes the action of Frobenius on the \'etale cohomology groups with torsion coefficients $\mathrm{H}^{i}(X, \mathbb{Z}/\ell \mathbb{Z})$ \footnote{we abuse notation by referring to the base change of $X$ to $X\times_{{\mathbb{F}_{q}}}\overline{\mathbb{F}}_{q}$, also as $X$} for primes $\ell= O(\log q)$, from which the zeta function of $X$, and thereby its point count can be recovered by a Chinese-remainder process. Our main result is the following.
\begin{theorem}
    There exists an algorithm that, on input $X$ as above, outputs $Z(X/\mathbb{F}_{q}, T)$ in time bounded by a polynomial in $\log q$.
\end{theorem}

\begin{remark}
    This theorem is restated in more detail as Theorem~\ref{thm:main} in Section~\ref{subsec:main} and proved therein. We in fact give an algorithm to compute the \'etale cohomology groups $\mathrm{H}^{i}(\mathcal{X}, \mu_\ell)$ with $\mathrm{Gal}(\overline{K}/K)$ - action in time polynomial in $\ell$, from which, for an input prime $\mathfrak{p}$, the local zeta function and point counts follow in polynomial time.
    
    \end{remark}

\subsection{Motivation}
Our work is fundamentally motivated by the following paraphrase of a question of Serre \cite[Preface]{serre}.
\begin{question}[Serre]
Is there an algorithm that, given a $\mathbb{Z}$ -- scheme $\mathcal{X}$ of finite type, computes the point count of its reduction $\#X(\mathbb{F}_{p})$ at any prime $p$ in time polynomial in $\log p$?   
\end{question}
In particular, this work solves the above question in the case $\dim X=2$, when $X$ is nice, at large enough primes of good reduction. In their book on computing the coefficients of the Ramanujan ${\tau} $ -- function, Couveignes and Edixhoven \cite[Epilogue]{ceramanujan} propose the existence of a strategy to count points on surfaces over finite fields, using the theory of Lefschetz pencils and d\'evissage; techniques which were used in Deligne's celebrated proof \cite{Weili} of the Weil conjectures. If realised, this would be an extension of polynomial-time counting methods from the dimension-one case of curves (and the conceptually similar case of abelian varieties) \cite{schoof, pila} to varieties of a higher dimension. 

An important motivation for these algorithms is computational evidence for conjectures in the {Langlands program} \cite{langlands}, a vast philosophy encompassing several areas of modern mathematics including number theory, representation theory and algebraic geometry. An object of study in part of the program, is the $L $ -- function of a variety $\mathcal{X}/\mathbb{Q}$, a conglomeration of the zeta functions at all the local factors. The Langlands-Rapoport conjecture \cite{lr}, in particular, gives the mod -- $p$ point-counts of Shimura varieties \footnote{algebraic varieties equipped with rich arithmetic data} a certain group-theoretic description.

Another angle of motivation is {diophantine geometry}, i.e., counting or classifying rational points on a variety $\mathcal{X}/\mathbb{Q}$. One approach towards this is computing the {Brauer-Manin obstruction} \cite{cts} (essentially measuring the failure of local-global principles) for specific varieties. This is defined using the Brauer group $\mathrm{H}^{2}(\mathcal{X}, \mathbb{G}_{m})$ of the variety in question, which is the \'etale cohomology in degree two, with coefficients in the multiplicative sheaf. With a view towards the diophantine setting, it would be prudent to have algorithms for the scenario over a finite field, with constant torsion coefficients to begin with. 

\subsection{Potential applications in computing}
A fundamental aspect of our work is the {explicitisation} of the {\'etale cohomology} of a surface, which should be viewed as an arithmetic or discrete analogue of the usual topological or Betti cohomology over the complex numbers. The latter notions do not translate easily to the setting over a finite field, and thus required the revolution of the Grothendieck school, thereby putting the Weil conjectures in proper context.


Our work lays the stepping stones toward solving a foundational problem for topological computation in the discrete setting, i.e., over finite fields. In particular, we, for the first time, make explicit (and give algorithms to compute) the \'etale cohomology groups with constant torsion coefficients $\mathrm{H}^{i}(X, \mu_\ell)$ of a nice surface $X$. This generalises to being able to compute with cohomology in degrees one and two, for varieties of higher dimension as well \cite{rsv, gcd}.

 The progenitor of point-counting algorithms, Schoof's algorithm \cite{schoof} for elliptic curves, paved the way for elliptic curve {cryptography}, which is ubiquitous today. In particular, it is necessary to run a point-counting algorithm to select a  curve suitable for cryptosystems. It is conceivable that our algorithms may come of use in efficiently designing cryptosystems around surfaces as well. Further, Brauer groups, mentioned earlier, arise naturally in the context of class field theory and homogeneous spaces, for which a general framework has been proposed with regard to applications to cryptography \cite{hhs}.




\subsection{Prior work \& special cases}
As mentioned earlier, the first advance in point-counting over finite fields came with Schoof's algorithm for elliptic curves. This was generalised to curves of higher genus and abelian varieties by Pila \cite{pila}. The cohomology groups in higher degree, however, have only recently been shown to be computable \cite{mad, ptv}.

  In Roy-Saxena-Venkatesh \cite{rsv}, a randomised algorithm was given to compute the factor $P_{1}(X/\mathbb{F}_{q}, T)$ for a nice variety $X$ of fixed degree, in time polynomial in $\log q$. Levrat has sketched a strategy to compute the full zeta function for surfaces \cite[IV.3.5, VI.4]{levrat} (see also \cite[\S 5]{leve}) based on the description of Couveignes-Edixhoven, but its runtime is exponential.


When the characteristic $p$ of the base field is fixed, the point-counting problem is essentially solved by Lauder-Wan \cite{LW} for varieties and Harvey \cite{harvey} for general arithmetic schemes by means of $p $ -- adic algorithms. As opposed to using \'etale cohomology, they feature $p $ -- adic trace formulas. These algorithms, however, have a runtime exponential in $\log p$.

\subsection{Obstructions in the prior techniques}
The main difficulty in counting points on surfaces in polynomial time so far, has been the lack of a concise representation of the \'etale cohomology groups $\mathrm{H}^{i}(X, \mu_\ell)$, particularly for $i=2$, on which the induced Frobenius action may be computed. In the approach of Levrat \cite{leve}, following Edixhoven, one reduces the computation of the group $\mathrm{H}^{2}(X, \mu_\ell)$ to the computation of $\mathrm{H}^{1}(V, \mu_\ell)$, where $V$ is a curve of genus polynomial in $\ell$. While algorithms are known to compute the first cohomology of curves \cite{huaier, couv}, their runtime is exponential in its genus. Thus for a prime $\ell$ of size $O(\log q)$, which is required for the intended Chinese remainder process, the above strategy implemented directly ends up giving an exponential-time algorithm.

Another approach would be to work directly with the Brauer group of $X$, whose $\ell $ -- torsion the group $\mathrm{H}^{2}(X, \mu_{\ell})$ captures. Elements in the Brauer group are, a priori, equivalence classes of Azumaya algebras; but it is not clear how one may obtain bounds to represent them, along with their group law and the equivalence relation they are subjected to.

\subsection{Proof ideas}
Our algorithm studies the \'etale cohomology of a surface by using the formalism of monodromy of vanishing cycles arising from a Lefschetz pencil. More specifically, we fibre the given surface $\mathcal{X}$ \footnote{in this part, we use the notation $\mathcal{X}$ to refer to the base change to the algebraic closure $\mathcal{X}\times_{K} \overline{K}$ as well} as a Lefschetz pencil of hyperplane sections, and then blow it up at the axis, yielding a morphism to $\mathbb{P}^{1}$. The cohomology of the blowup $\tilde{\mathcal{X}}$\footnote{which we now call $\mathcal{X}$, and is equipped with a morphism $\pi:\mathcal{X}\rightarrow \mathbb{P}^{1}$} can be understood using the sequence (\ref{seq}) coming from the Galois cohomology of the tame fundamental group of the line with the critical locus (i.e., the finite set $\mathcal{Z}=\mathbb{P}^{1}\setminus \mathcal{U}$ where the fibres are nodal) removed. In particular, one needs to be able to compute the monodromy action on the cohomology of the generic fibre. 

\hfill

Our solution is to first compute the $\ell $ -- division polynomial system (the zero dimensional ideal whose roots are the distinct $\ell$ -- torsion points) for the torsion in the Jacobian of the generic fibre, and view the choice of a  cospecialisation morphism at a singular point $z$ as picking a Puiseux series expansion around $z$. Working in characteristic zero, we compute the local monodromy using this Puiseux expansion. Additionally, we identify the vanishing cycle $\delta_z$ at $z$ using an auxiliary smooth point $u_z$ within the radii of convergence of the Puiseux expansions around $z$ combined with numerical/diophantine approximation methods in a technique we call `re-centering'. Specifically, we also compute each vanishing cycle as an element in the cohomology $\mathcal{F}_{\overline{\eta}}$ of the generic fibre, where $\mathcal{F}=\mathrm{R}^{1}\pi_{\star}\mu_{\ell}$ is the first derived pushforward on $\mathbb{P}^{1}$.

\hfill 


Following this, we move to the \'etale open cover $\mathcal{V}\rightarrow \mathcal{U}$ trivialising the locally constant sheaf $\mathcal{F}\vert_{\mathcal{U}}=\mathrm{R}^{1}\pi_{\star}\mu_{\ell}\vert_{\mathcal{U}}$ on $\mathcal{U}$. Call $\mathcal{E}\subset \mathcal{F}\vert_{\mathcal{U}}$ the locally constant subsheaf of vanishing cycles on $\mathcal{U}$. The normalisation of $\mathbb{P}^{1}$ in the function field of $\mathcal{V}$ yields a morphism of smooth projective curves $\tilde{\mathcal{V}}\rightarrow \mathbb{P}^{1}$ ramified exactly at $\mathcal{Z}$. Calling the representation $\rho_{\ell}:\pi_{1}(\mathcal{U}, \overline{\eta})\rightarrow \mathrm{Aut}(\mathcal{E}_{\overline{\eta}})$, we write $G:=\mathrm{im}(\rho_\ell)$, and note that the cover $\mathcal{V}\rightarrow \mathcal{U}$ has Galois group $G$. The group $G$ acts naturally on $\tilde{\mathcal{V}}$ via automorphisms, which extends to an action on $\mathrm{H}^{1}(\tilde{\mathcal{V}}, \mu_\ell)\simeq \mathrm{Jac}(\tilde{\mathcal{V}})[\ell]$. 

\hfill

Further, to compute the part of $\mathrm{H}^{2}(\mathcal{X}, \mu_\ell)$ corresponding to $\mathrm{H}^{1}(\mathbb{P}^{1}, \mathcal{F})$, it suffices to compute the invariant subspace of $\mathrm{H}^{1}(\tilde{\mathcal{V}}, \mu_\ell)\otimes_{\mathbb{Z}/\ell \mathbb{Z}}\mathcal{E}_{\overline{\eta}}$, under the diagonal action of $G$. This is done by choosing an auxiliary prime $\mathfrak{P}$ with characteristic distinct from $\ell$ and of size $O(\ell)$, of good reduction and isolating the subspace spanned by the images of all $G$- equivariant homomorphisms from $\mathcal{E}_{\overline{\eta}}^{\vee}$ to $\mathrm{Jac}(\tilde{\mathcal{V}}_{\mathfrak{P}})[\ell]$ (which we call the mod-$\mathfrak{P}$ Edixhoven subspace) in the cohomology of the reduced curve. We then $\mathfrak{P}$ -- adically lift the concerned subspace to the char zero Edixhoven subspace $\mathbb{E}$, using work of Mascot \cite{mascothensel} on Hensel-lifting torsion points. With this, the arithmetic Galois action follows, along with zeta function and point counts, for large primes of good reduction.

\subsection{Leitfaden}
Section~\ref{sec:prelim} delineates the cohomological preliminaries that form the fundamental basis of our algorithms. Section~\ref{sec:sub} develops subroutines including Weil pairings and Puiseux expansions for vanishing cycles, which are used in the algorithms of Section~\ref{sec:algos}. The main theorem is proved in Section~\ref{subsec:main}. Complexity analyses of all algorithms are provided in Section~\ref{sec:comp}. The appendices in order include material on recovering the zeta function, background on height theory, a recap of certain results of Igusa, and a known algorithm for computing equations of Jacobians due to Anderson.


\section{Cohomological preliminaries}
\label{sec:prelim}
The aim of this section is to compile standard background material on the cohomology of the various varieties that will be required for the algorithm. We present cohomology computations when explicitly known, and point to the existence of algorithms in the curve case: smooth, nodal, and for a smooth curve over the rational function field.

\subsection{Cohomology of a surface}
\label{subsec:etsurf}
 In this subsection, we briefly recall cohomology computations for surfaces. A standard reference is \cite[V.3]{milne80}. Let $k$ be a separably closed field and let ${X}$ be a smooth, projective geometrically irreducible surface over it. Following \cite[Algorithm 3]{rsv}, one may fibre $X$ as a Lefschetz pencil $\pi: \tilde{X}\rightarrow \mathbb{P}^{1}$ of hyperplane sections over the projective line, where $\tilde{X}$ is the surface obtained by blowing up $X$ at the axis $\Upsilon$ of the pencil. Denote $Z\subset \mathbb{P}^{1}$ the finite critical locus, whose corresponding fibres have exactly one node (with $\#Z=r$) and let $U=\mathbb{P}^{1}\setminus Z$ be the locus of smooth fibres. Let $\ell$ be a prime distinct from the characteristic of $k$ and write  $\mathcal{F}:=\mathrm{R}^{1}\pi_{\star}\mu_{\ell}$ for the constructible derived push-forward sheaf on $\mathbb{P}^{1}$. We note that the restriction $\mathcal{F} \vert_{ U}$ is a locally constant sheaf (or local system) on $U$. Let $\overline{\eta}\rightarrow \mathbb{P}^{1}$ be a geometric generic point and let $g$ denote the genus of the generic fibre $X_{\overline{\eta}}$, viewed as a curve over the function field of the projective line. Firstly, one recalls \cite[Lemma 33.2]{milnelec}
 \begin{equation}
 \mathrm{H}^{i}(\tilde{X}, \mathbb{Q}_{\ell})\simeq \begin{cases} \mathrm{H}^{i}(X, \mathbb{Q}_{\ell}), \ i\ne 2; \\
     
 \mathrm{H}^{2}(X, \mathbb{Q}_{\ell})\oplus \mathrm{H}^{0}(\Upsilon\cap X, \mathbb{Q}_{\ell})(-1), \ i=2
 
 \end{cases}
 \end{equation}
 so it suffices to compute the zeta function of $\tilde{X}$ (see Section~\ref{app:rec}). In Algorithm~\ref{algo:blowup}, we detail a method to compute equations for the blowup.

\begin{algorithm}
    \caption{\texttt{Blowup of a surface at a point}}
    \label{algo:blowup}

    \begin{itemize}
        \item \textbf{Input:} A nice surface $X\subset \mathbb{P}^{N}$ presented as homogeneous forms $f_{1}, \ldots , f_{m}$ and a point $P\in X$. Assume without loss, $P=[0:0:\ldots :1]$.
        \item \textbf{Output:} A surface $\tilde{X}$ that is the blowup of $X$ at $P$ and a morphism $\pi:\tilde{X}\rightarrow X$
    \end{itemize}
    \begin{algorithmic}[1]
    \STATE Consider the projection $\varphi_{P}: \mathbb{P}^{N}\setminus P\rightarrow \mathbb{P}^{N-1}$ from $P$.
    \STATE The blowup $\tilde{X}$ of $X$ at $P$ is given by the closure in $X\times \mathbb{P}^{N-1}$ of the graph of $\varphi_P$ restricted to $X\setminus P$.
    \STATE Use the Segre embedding to obtain equations for $\tilde{X}$.
    \STATE The morphism $\pi:\tilde{X}\rightarrow X$ is obtained by projection to the first factor.
        
    \end{algorithmic}
\end{algorithm}
Henceforth, without loss of generality, we may assume $X$ may be fibred as $\pi:X\rightarrow \mathbb{P}^{1}$ as a Lefschetz pencil of hyperplane sections. From the L\'eray spectral sequence $$
\mathrm{H}^{i}(\mathbb{P}^{1}, R^{j}\pi_{\star}\mu_{\ell})\Rightarrow \mathrm{H}^{i+j}(X, \mu_{\ell}),
$$
one has
\begin{equation}
\mathrm{H}^{i}(X, \mu_{\ell})\simeq \begin{cases}
    \mu_{\ell}, \ i=0; \\
    \mathrm{H}^{0}(\mathbb{P}^{1}, \mathcal{F}), i=1;\\
    \mathrm{H}^{1}(\mathbb{P}^{1}, \mathcal{F})\oplus\langle \gamma_{E}\rangle\oplus \langle \gamma_{F}\rangle , \ i=2;\\
    \mathrm{H}^{2}(\mathbb{P}^{1}, \mathcal{F}), \ i=3; \\
    \mu_{\ell}^{\vee},\  i=4; \\
    0, \ i>4.
    
\end{cases}
\end{equation}
 Here $\gamma_{E}$ and $\gamma_{F}$ are certain cycle classes on $X$ (viewed in $\mathrm{H}^{2}$ via the cycle class map) corresponding to the class of a section of $\pi$ and the class of a smooth fibre of $\pi$ respectively. One needs to work more to make the above groups explicit.  
\\ \\
Recall the theory of {vanishing cycles} on a surface \cite[3.1, 3.2]{rsv}. For each $z\in {Z}$, one obtains a mod -- $\ell$ vanishing cycle  $\delta_{z}$ at $z$ as the generator of the kernel of the map $\mathrm{Pic}^{0}(X_{z})[\ell]\rightarrow \mathrm{Pic}^{0}(\widetilde{X}_{z})[\ell]$ induced by the normalisation $\widetilde{X}_{z}\rightarrow X_{z}$. Using a {cospecialisation map}\footnote{which depends on the choice of an embedding of the strict henselisation $\widehat{\mathcal{O}}_{\mathbb{P}^{1}, z}\hookrightarrow k(\overline{\eta})$, see Section~\ref{subsec:cospec}} 

\begin{equation}
\label{eqn:cospec}
\phi_{z_{j}}:\mathcal{F}_{z_{j}}\hookrightarrow \mathcal{F}_{\overline{\eta}}
\end{equation}
for each $z_{j}\in \mathcal{Z}$, one obtains the subspace generated by all the vanishing cycles $\delta_{z_{j}}$ in $\mathcal{F}_{\overline{\eta}}$. The geometric \'etale fundamental group $\pi_{1}(U, \overline{\eta})$ acts on $\mathcal{F}_{\overline{\eta}}$, factoring through the tame quotient $\pi_{1}^{\mathfrak{t}}(U, \overline{\eta})$, via the Picard-Lefschetz formulas. In particular, $\pi_{1}^{\mathfrak{t}}(U, \overline{\eta})$ is generated topologically by $\#Z=r$ elements $\sigma_{j}$ satisfying the relation $\prod_{j}\sigma_{j}=1$. We have for $\gamma\in \mathcal{F}_{\overline{\eta}}$
 \begin{equation}
 \label{eqn:piclef}
 \sigma_{j}(\gamma)= \gamma-\epsilon_{j}\cdot \langle \gamma, \delta_{z_{j}} \rangle\cdot \delta_{z_{j}},
 \end{equation}
 where $\langle \cdot, \cdot \rangle$ denotes the Weil pairing on $\mathrm{Pic}^{0}(X_{\overline{\eta}})[\ell]$ and for a uniformising parameter $\theta_{j}$ at $z_{j}$, one has $\sigma_{j}(\theta_{j}^{1/\ell})=\epsilon_{j}\cdot\theta_{j}^{1/\ell}$. Further, $\sigma_j$ is understood as the canonical topological generator for the tame inertia $I^{\mathfrak{t}}_{z_j}$ at $z_j$ (after having made consistent choices for primitive roots of unity).

 One sees immediately that the monodromy \footnote{action of the \'etale fundamental group on $\mathcal{F}_{\overline{\eta}}$} is symplectic, i.e., the representation $$\rho: \pi_{1}^{\mathfrak{t}}(U, \overline{\eta})\longrightarrow \mathrm{GL}(2g, \mathbb{F}_{\ell})$$ has image in $\mathrm{Sp}(2g, \mathbb{F}_{\ell})$, the group of symplectic transformations of the vector space $\mathbb{F}_{\ell}^{2g}$, as it has to preserve the Weil pairing on $\mathcal{F}_{\overline{\eta}}$. \\ 
 
 Next, one recalls the following complex, \cite[Theorem 3.23]{milne80} coming from the Galois cohomology of $\pi_{1}^{\mathfrak{t}}(U, \overline{\eta})$
\begin{equation}
\label{seq}
    \mathcal{F}_{\overline{\eta}}\xrightarrow{\alpha} \left(\mathbb{Z}/\ell \mathbb{Z}\right)^{r}\xrightarrow{\beta}\mathcal{F}_{\overline{\eta}}
\end{equation}
with, for any $\gamma\in \mathcal{F}_{\overline{\eta}}$
$$
\alpha(\gamma)= (\langle \gamma, \delta_{z_{1}}\rangle, \ldots , \langle \gamma, \delta_{z_{r}}\rangle)$$ and for any $r $ -- tuple $(a_{1}, \ldots , a_{r})\in \left(\mathbb{Z}/\ell \mathbb{Z}\right)^{r}$
$$
\beta(a_{1}, \dots , a_{r})=a_{1}\cdot \delta_{z_{1}}+ a_{2}\cdot \sigma_{1}(\delta_{z_{2}})+\ldots + a_{r}\cdot \left(\prod_{j=1}^{r-1}\sigma_{j}\right)(\delta_{z_{r}}).
$$
  The cohomology groups of the above complex are related to the cohomology of $X$, i.e.,
\begin{equation}
    \mathrm{H}^{i}(X, \mu_\ell)\simeq \begin{cases}
        \ker(\alpha), \ i=1; \\
        \left(\ker(\beta)/\mathrm{im}(\alpha)\right) \oplus < \gamma_{E}> \oplus <\gamma_{F}>, \ i=2; \\
        \mathrm{coker}(\beta), \ i=3.
    \end{cases}
\end{equation}
In particular, we have that $\mathrm{H}^{1}(\mathbb{P}^{1}, \mathcal{F}) \simeq \ker(\beta)/\mathrm{im}(\alpha)$. If the situation is over a finite field, it is sufficient to compute the action of the Frobenius $F_{q}^{\star}$ on $\mathrm{H}^{1}(\mathbb{P}^{1}, \mathcal{F})$ as it acts as `multiplication by $q$' on $<\gamma_{E}>$ and $<\gamma_F>$. More generally, the Galois action on $<\gamma_E>$ and $<\gamma_F>$ is via the cyclotomic character.


\subsection{Cohomology of a smooth fibre}
\label{subsech:sm}
Let $X_{u}$ be a smooth fibre of the Lefschetz pencil $\pi: X\rightarrow \mathbb{P}^{1}$ at a point $u\in {U}$. The objective of this section is to state how to compute and efficiently represent the $\ell$ -- torsion in the Jacobian of $X_{u}$, i.e., the group $\mathrm{Pic}^{0}(X_{u})[\ell]\simeq (\mathbb{Z}/\ell\mathbb{Z})^{2g}$. Algorithms for this procedure are known, see e.g., \cite{huaier} and \cite{pila}. The two are markedly different, in that the former works with the Jacobian by means of divisor arithmetic whereas the latter requires an explicit embedding of the Jacobian including equations and addition law. We use both for different applications.  
\begin{remark}
Over a finite field, knowing the zeta function of $X_{u}$, an algorithm of Couveignes \cite[Theorem 1]{couv} also computes $\mathrm{Pic}^{0}(X_{u})[\ell]$, but any (known) algorithm that computes $Z(X_{u}/ \mathbb{F}_{Q}, T)$ in time poly$(\log Q)$ also computes the $\ell$ -- torsion in the Jacobian for small primes $\ell$ first as a subroutine.
\end{remark}

\begin{theorem}[Arithmetic on Jacobians via divisors]
\label{thm:jacarith}
    Given a curve $C$ of genus $g$ over an effective field $k$, and a divisor $E$ on $C$ of degree $d$, there exists an algorithm that computes a basis for the Riemann-Roch space $\mathcal{L}(E)$ in time
    $$
    \mathrm{poly}(g\cdot d).
    $$
    Moreover, arithmetic on $\mathrm{Pic}^{0}(C)$ can be performed in polynomial time.
\end{theorem}
\begin{proof}
    Apply \cite{hirr} or \cite{lespa} for computing Riemann-Roch spaces. Divisor arithmetic on the Jacobian can be done using \cite{khuri1, khuri}.
\end{proof}

\begin{theorem}[Huang-Ierardi]
\label{thm:huaier}
    Let $C\subset \mathbb{P}^{N}$ be a smooth, projective curve of genus $g$ over an effective field $k$ and let $\ell$ be  a prime distinct from the characteristic of $k$. There exists an algorithm to compute $\mathrm{Pic}^{0}(C)[\ell]$ via divisor representatives in time $\mathrm{poly}(\ell)$. If $k=\mathbb{F}_{q}$ is a finite field, the complexity is polynomial in $\log q$ as well.
\end{theorem}
\begin{proof}
    See \cite[\S 5]{huaier}.
\end{proof}

\begin{theorem}[Pila]
\label{thm:pila}
    Let $C\subset \mathbb{P}^{N}$ be a smooth, projective curve of genus $g$ over an effective field $k$ and let $\ell$ be  a prime distinct from the characteristic of $k$. Assume $\mathrm{Pic}^{0}(C)=\mathrm{Jac}(C)$ is provided as an abelian variety via homogeneous polynomial equations in $\mathbb{P}^{M}$ along with addition law. Then, there exists an algorithm to compute the points representing $\mathrm{Pic}^{0}(C)[\ell]$ in $\mathbb{P}^{M}$ in time polynomial in $\ell$.  If $k=\mathbb{F}_{q}$ is a finite field, the complexity is polynomial in $\log q$ as well.
\end{theorem}
\begin{proof}
    See \cite[\S 2, \S 3]{pila}.
\end{proof}

\subsection{Cohomology of a nodal fibre}
\label{subsec:nod}
Let $X_{z}$ be a nodal curve, obtained as a critical fibre of the Lefschetz pencil in the previous section. The objective of this section is to state how we may represent and compute the cohomology  $\mathrm{H}^{1}(X_{z}, \mu_{\ell})\simeq \mathrm{Pic}^{0}(X_{z})[\ell]\simeq (\mathbb{Z}/\ell \mathbb{Z})^{2g-1}$ concisely. Let $\widetilde{X}_{z}\rightarrow X_{z}$ be the normalisation of this nodal curve. Let $P_{z}\in X_{z}$ denote its singularity and let $D_{z}=Q_{z}+R_{z}$ denote the exceptional divisor on $\widetilde{X}_{z}$, where $Q_{z}, R_{z}\in \widetilde{X}_{z}$. It is possible to describe $\mathrm{Pic}^{0}(X_{z})$ in terms of $\mathrm{Pic}^{0}(\widetilde{X}_{z})$ and $D_{z}$. First, write $$\mathrm{Div}_{D_{z}}(\widetilde{X}_{z}):=\mathrm{Div}(\widetilde{X}_{z}\setminus \{Q_{z}, R_{z}\})$$ and let $k(\widetilde{X}_{z})$ denote the function field of $\widetilde{X}_{z}$. For $f\in k(\widetilde{X}_{z})^{*}$, we say $$f\equiv 1 \ \mathrm{mod} \ D_{z} \ \ \text{if} \ \ v_{Q_{z}}(1-f)\geq 1 \ \ \text{and} \ \ v_{R_{z}}(1-f)\geq 1.$$  Define
\begin{equation}
\mathrm{Pic}^{0}_{D_{z}}(\widetilde{X}_{z}):=\mathrm{Div}^{0}_{D_{z}}(\widetilde{X}_{z})/\langle \{ \mathrm{div}(f) \ \vert \ f\equiv 1 \mod D_{z}\}\rangle.
\end{equation}
Then, it is possible to show \cite[Chapter V]{serrealg}\footnote{see also \cite[Lemma 2.3.8]{levrat}} that $\mathrm{Pic}^{0}(X_{z})\simeq \mathrm{Pic}^{0}_{D_{z}}(\widetilde{X}_{z})$. In particular, we have 
\begin{equation}
\label{eqn:picc}
\mathrm{Pic}^{0}(X_{z})[\ell]\simeq \mathrm{Pic}^{0}_{D_{z}}(\widetilde{X}_{z})[\ell].
\end{equation}
 The upshot is that we may also represent the elements (and group law) of the LHS in the isomorphism~\ref{eqn:picc}, using effective Riemann-Roch algorithms on the normalisation. In particular, one can isolate the subspace generated by the vanishing cycle at $z$, namely $\langle\delta_{z}\rangle\subset \mathrm{Pic}^{0}(X_{z})[\ell]$, as the kernel of the natural induced map
 $$\mathrm{Pic}^{0}_{D_{z}}(\widetilde{X}_{z})[\ell]\longrightarrow \mathrm{Pic}^{0}(\widetilde{X}_{z})[\ell].$$
 \begin{remark}
   We may compute the elements of $\mathrm{Pic}^{0}(X_z)[\ell]$ via specialisation to $z$ of the ideal $\prescript{(\ell)}{}{}\mathcal{I}_{\overline{\eta}}$ computing the $\ell$ --  torsion in the generic fibre using Algorithm~\ref{algo:genelldiv}. By a result of Igusa \cite[Theorem 3]{igusa}, we know that the $\overline{k}$ -- roots of this specialisation contain the $\ell^{2g-1}$ torsion elements of the generalised Jacobian $\mathrm{Pic}^{0}(X_z)[\ell]$. The other roots correspond to singularities of the completion of the generalised Jacobian $\mathrm{Pic}^{0}(X_z)$ by Theorem~\ref{thm:igtors}.
\end{remark}
 
 It requires more work to completely identify the vanishing cycle $\delta_{z}$ (upto sign), this is done in Section~\ref{sec:sub} using the Picard-Lefschetz formulas (\ref{eqn:piclef}).

\subsection{Cohomology of the generic fibre}
\label{subsec:gen}
As a result of the Lefschetz fibration $\pi:X\rightarrow \mathbb{P}^{1}$, we may think of the surface $X$ as defining a relative curve over $k(t)$, the function field of the projective line. We refer to this curve as the `generic fibre' of the pencil, $X_{\overline{\eta}}$. Scheme-theoretically, this corresponds to the fibre of $\pi$ over a geometric generic point $\overline{\eta}\rightarrow \mathbb{P}^{1}$. The stalk $\mathcal{F}_{\overline{\eta}}\simeq \mathrm{Pic}^{0}(X_{\overline{\eta}})[\ell]$ is the $\ell$ -- torsion in the Jacobian of this relative curve of genus $g$. \footnote{The genus of any smooth fibre over $u\in {U}$ will also be $g$.} 

The main objective of this section is to describe a zero-dimensional radical ideal $\prescript{(\ell)}{}{}\mathcal{I}_{\overline{\eta}}$ over $k(t)$\footnote{i.e., one-dimensional over $k$}, whose $\overline{k(t)}$ -- roots correspond exactly to elements of $\mathcal{F}_{\overline{\eta}}$. First, we bound the degree of this system. We know that $\mathcal{F}_{\overline{\eta}}\simeq (\mathbb{Z}/\ell \mathbb{Z})^{2g}$ as an abelian group, so the system has  $\ell^{2g}$ -- many $\overline{k(t)}$ -- roots. It remains to bound the degree of the system in $t$, i.e., the degree of the polynomials in $t$ occurring as coefficients of the above system. First, we note by \cite[\S 4.2]{rsv} 
\begin{equation}
\#Z\leq D^{N+1} \ \ \ \ \text{and} \ \ \ \ g\leq D^{2}-2D+1.
\end{equation}
Next, denote by $\kappa$ the minimal Galois extension of ${\overline{k}(t)}$ that all the elements of $\mathcal{F}_{\overline{\eta}}$ can be defined over. We know that the extension $\kappa/\overline{k}(t)$ has its Galois group as a subgroup of $\mathrm{Sp}(2g, \mathbb{F}_{\ell})$, so in particular, its degree is bounded above by $\ell^{4g^{2}}$. Further, we see that the curve $V$ obtained by normalising the function field of $U$ in $\kappa$ gives an \'etale cover $V\rightarrow U$ which trivialises the locally constant sheaf $\mathcal{F}\vert_{U}$ to a {constant} sheaf $\mathcal{G}$ on $V$. More specifically, $V$ is a cover of $\mathbb{P}^{1}$ of degree bounded by $\ell^{4g^{2}}$, tamely ramified at $Z$. Therefore, the product $$\#Z\cdot \ell^{4g^{2}}\leq D^{N+1}\ell^{4(D+1)^{4}}$$ which is polynomial in $\ell$, serves as an upper bound for the genus $g_{V}$ of $V$\footnote{by the Riemann-Hurwitz formula}; and hence, also for the complexity of the system $\prescript{(\ell)}{}{}\mathcal{I}_{\overline{\eta}}$ in the variable $t$. 

\begin{remark}
    Mascot \cite[Algorithm 2.2]{mascot} also proposes an algorithm to compute $\ell$ -- division polynomials for the Jacobian of a curve over $\mathbb{Q}(t)$, based on $(p', t)$ -- adically lifting torsion points for a small, auxiliary prime $p'$. It is however mentioned \cite[Remark 4.3]{mascot} that parts of his algorithm are not rigorous.
\end{remark}

\begin{algorithm}
    \caption{\texttt{Computing the} $\ell$ -- \texttt{division ideal of} $\mathrm{Pic}^{0}({X}_{\overline{\eta}})$}
    \label{algo:genelldiv}
     \begin{itemize}
        \item \textbf{Input:} A Lefschetz pencil $\pi:X\rightarrow \mathbb{P}^{1}$.
        
        \item \textbf{Output:} A radical ideal $\prescript{(\ell)}{}{}\mathcal{I}_{\overline{\eta}}$ over $k(t)$ whose $\overline{k(t)}$ -- roots correspond to the $\ell $ -- torsion points of $\mathrm{Pic}^{0}(X_{\overline{\eta}})$.
    \end{itemize}

    \begin{algorithmic}[1]

    \STATE Compute equations for $\mathrm{Pic}^{0}({X}_{\overline{\eta}})=\mathrm{Jac}({X}_{\overline{\eta}})$ using Theorem~\ref{thm:eqnjac}, realising it as a subvariety of $\mathbb{P}^{M}$. 

    \STATE Compute the multiplication by $\ell$ -- map as a morphism on $\mathrm{Pic}^{0}({X}_{\overline{\eta}})$ by Theorem~\ref{thm:eqnjac}.
    
    \STATE Compute the equations for the pre-image of the identity element of the Jacobian.

    \STATE Return the ideal $\prescript{(\ell)}{}{\mathcal{I}_{\overline{\eta}}}$ so obtained.
\end{algorithmic}
\end{algorithm}
\begin{remark}
    Algorithm~\ref{algo:genelldiv} also provides an algorithm to compute the $\ell$ -- division ideal corresponding to $\mathrm{Pic}^{0}({X}_u)$ for a smooth $u\in {U}$ by simply specialising $\prescript{(\ell)}{}{}\mathcal{I}_{\overline{\eta}}$ to $u$. 
\end{remark}

\section{Essential subroutines}
\label{sec:sub}

In this section, we compute and explicitly present the monodromy representation of the \'etale fundamental group associated to the sheaf of vanishing cycles. Specifically, we recall pairing algorithms and Puiseux series to construct the cospecialisation maps at singular points, and specialisation to smooth points with the final motive of computing the monodromy action on the cohomology of the generic fibre. As a by product, we also explcitly compute local monodromy, and the vanishing cycle at each singular point.



\subsection{Pairing}
\label{subsec:pairing}
Now, we define the Weil pairing on the $\ell$ -- torsion points on the Jacobian of a curve and delineate an efficient algorithm to compute it.

\begin{definition}
Let $C$ be a smooth projective curve over an algebraically closed field $k$, let $J$ be its Jacobian and let $\ell$ be a prime number. The $\mathrm{mod}$ -- $\ell$ Weil pairing on $J$ is a map
\begin{equation*}
    J[\ell]\times J[\ell]\longrightarrow \mu_{\ell}
\end{equation*}
given by
\begin{equation*}
    (D_1, D_2)\mapsto \langle D_1, D_2\rangle. 
\end{equation*}
Let $\ell\cdot D_1=\mathrm{div}(f)$ and $\ell\cdot D_2=\mathrm{div}(g)$ for $f,g \in k(C)^{*}$. Then, $\langle D_1, D_2\rangle = \frac{f(D_2)}{g(D_1)}$.
\end{definition}

\begin{theorem}
\label{thm:pairing}
    There exists an algorithm, that, on input a smooth, projective curve $C$ over $\mathbb{F}_{q}$, a prime number $\ell$ coprime to $q$, two $\ell$ -- torsion divisors $D_{1}, D_{2}\in \mathrm{Pic}^{0}(C)[\ell]$, computes the Weil pairing $\langle D_{1}, D_{2} \rangle $ in time   $$
    \mathrm{poly}(\log q \cdot \ell).
    $$
\end{theorem}
\begin{proof}
    See \cite[\S 16.1]{handbook} or \cite[Lemma 10]{couv}.
\end{proof}

\begin{algorithm}
    \caption{\texttt{Computing the Weil pairing}}
    \label{algo:pairing}
    \begin{itemize}
        \item \textbf{Input:} A smooth projective curve $C$ over $\mathbb{F}_{q}$ and two divisors $D_{1}, D_{2}\in \mathrm{Pic}^{0}(C)[\ell]$.
        \item \textbf{Output:} The value $\langle D_{1}, D_{2}\rangle \in \mu_{\ell}(\overline{\mathbb{F}}_{q})$.
    \end{itemize}

    \begin{algorithmic}[1]
    \STATE Find $f, g \in k(C)^{*}$ such that $\mathrm{div}(f)=\ell \cdot D_{1}$ and $\mathrm{div}(g)=\ell \cdot D_{2}$ using an effective Riemann-Roch algorithm from Theorem~\ref{thm:jacarith}.
     \STATE Evaluate $\frac{f(D_{2})}{g(D_{1})}$ using \cite[Lemma 10]{couv}.

     \STATE Return the value of $\frac{f(D_{2})}{g(D_{1})}$.
    \end{algorithmic}
\end{algorithm}

\begin{remark}
    While the algorithm from \cite{couv} runs with stated complexity over a finite field, it works over a number field as well, with similar dependence on $\ell$. We note that for a curve $C$ over a number field $K$, the $\ell$ -- torsion is defined over an extension $K'$ of $K$ of degree a polynomial in $\ell$ as $\mathrm{Gal}(K'/K)\subset \mathrm{GL}(2g, \mathbb{F}_{\ell})$, where $g$ is the genus of $C$. The height of the $\ell$ -- torsion elements is bounded, by Theorem~\ref{thm:tors_height}. Additionally, we note that there are also pairing algorithms running in time polynomial in $\ell$ that work directly with an embedding of the Jacobian of the curve. See \cite{lr2, lrpairing}.
\end{remark}

\subsection{Cospecialisation at a singular fibre}
\label{subsec:cospec}
In this subsection, we make the cospecialisation maps (\ref{eqn:cospec}) from the cohomology of a special fibre to that of the generic fibre, explicit.

\hfill

Let $\pi:\mathcal{X}\rightarrow \mathbb{P}^{1}$ be a Lefschetz pencil of hyperplane sections on a nice surface over a number field $K$. We fix an embedding $\overline{K}\hookrightarrow \mathbb{C}$ at the outset. Denote by $\mathcal{Z}\subset \mathbb{P}^{1}$ the finite subset parametrising the critical (nodal) fibres and write $\mathcal{U}=\mathbb{P}^{1}\setminus \mathcal{Z}$. Denote by $\mathcal{F}:=\mathrm{R}^{1}\pi_{\star}\mu_{\ell}$, the first derived pushforward sheaf on $\mathbb{P}^{1}$ and let $\overline{\eta}\rightarrow \mathbb{P}^{1}$ be a geometric generic point. Let $z\in \mathcal{Z}$. Consider the strictly Henselian ring $\widehat{\mathcal{O}}_{\mathbb{P}^{1}, z}$. By \cite[Proposition 4.10]{milnelec}, it can be understood as the elements of $$\overline{K}[[t-z]]\cap \overline{K(t)},$$
i.e., those power series in $t-z$ which are algebraic over $\overline{K}(t)$. Let $\mathtt{K}_z$ denote a separable closure of the field of fractions of $\widehat{\mathcal{O}}_{\mathbb{P}^{1}, z}$. After \cite[\S 20]{milnelec}, we know that the choice of an embedding $\mathtt{K}_{z}\hookrightarrow \overline{K(t)}$ determines the cospecialisation morphism $$\phi_{z}:\mathcal{F}_{z}\hookrightarrow \mathcal{F}_{\overline{\eta}}.$$
In particular, this choice is the \'etale analogue of a path or `\textit{chemin}'.
We begin with the following. 
\begin{definition}[Puiseux series]
    Let $\mathbb{K}$ be a field. A formal \textit{Puiseux series} $f(t)$ over $\mathbb{K}$ in the variable $t$ is an expression of the form
    $$
    f(t)=\sum_{j\geq M}^{\infty}a_{j}t^{j/n}
    $$
     for some $M\in \mathbb{Z}$, $n\in \mathbb{Z}_{>0}$ and $a_{j}\in \mathbb{K}$. The field of formal Puiseux series is denoted $\mathbb{K}\langle \langle t \rangle \rangle$. In particular, we have
     $$
     \mathbb{K}\langle \langle t\rangle\rangle=\bigcup_{n=1}^{\infty}\mathbb{K}((t^{1/n})),
     $$
     where $\mathbb{K}((t))$ is the field of formal Laurent series in $t$ with coefficients in $\mathbb{K}$. It is a classical result that if $\mathbb{K}$ is algebraically closed of characteristic zero, then $\mathbb{K}\langle \langle t \rangle \rangle$ is the algebraic closure of $\mathbb{K}((t))$.
\end{definition}

We notice that the field $\overline{K}\langle \langle t-z \rangle \rangle$ of Puiseux series in $t-z$,  contains both $\mathtt{K}_{z}$ and a copy of $\overline{K(t)}$, so we seek to fix the stated embedding therein. We are only concerned with the finite field extension $\mathbf{K}$ of $K(t)$ that all the points of $\mathrm{Pic}^{0}(\mathcal{X}_{\overline{\eta}})[\ell]$ are defined over. It is the splitting field of the $\ell$ -- division ideal $\prescript{(\ell)}{}{}\mathcal{I}$ of $\mathrm{Pic}^{0}(\mathcal{X}_{\overline{\eta}})$ computed in Section~\ref{subsec:gen}. We observe
\begin{equation}
\label{eqn:gl}
[\mathbf{K}: K(t)]\leq \#\mathrm{GL}(2g, \mathbb{F}_{\ell}),
\end{equation}
where $g$ is the genus of $\mathcal{X}_{\overline{\eta}}$. Therefore, we may write $\mathbf{K}=K(t)\left(\bm{\tau}\right)$, where $\bm{\tau}$ is a primitive element for $\mathbf{K}/K(t)$. By (\ref{eqn:gl}), we may assume $\bm{\tau}$ has a minimal polynomial $\mu(x)$ with coefficients in $K(t)$, of degree bounded by a polynomial in $\ell$. The height of the coefficients can also be assumed to be bounded by a polynomial in $\ell$ by Section~\ref{app:height}. In order to fix an embedding $\mathbf{K}\hookrightarrow \overline{K}\langle \langle t-z \rangle \rangle$, we simply pick a Puiseux series expansion $\lambda_z$ of $\bm{\tau}$ in $t-z$, as a root of $\mu(x)$. This is made possible using the following classical theorem-algorithm due to Newton and Puiseux.
\begin{theorem}[Newton-Puiseux]
\label{thm:ntpu}
    Let $\mu(x, t)=0$ be a curve in $\mathbb{C}^{2}$. Let $d_x$ be the degree of $\mu$ in the variable $x$. Then, around any $u\in \mathbb{C}$, there exist $d_x$ many Puiseux expansions
    $$
    x_i(t)=\sum_{j\geq M}^{\infty} \alpha_{i, j}(t-u)^{j/N}
    $$
    satisfying $\mu(x_i, t)=0$. Each $x_i(t)$ converges for values of $t$ in an open neighbourhood of $u$. Moreover, given a positive integer $m$, there exists an algorithm that outputs the first $m$ coefficients of all the expansions of $x_i$ in time $$\mathrm{poly}(d_x\cdot m).$$
\end{theorem}
\begin{proof}
    For the existence, see \cite[Theorem 2.1]{wall}. The algorithm with stated complexity is from \cite[Theorem 1]{walsh}.
\end{proof}
\begin{remark}
We see that if $\lambda(t)=\sum_j \alpha_j t^{j/M}$ is an algebraic Puiseux series as a solution of $\mu(x,t)=0$, so are its conjugates $\sum_j \alpha_j\zeta_{M}^{ij}t^{j/M}$, for $\zeta_M$ a primitive $M^{\mathrm{th}}$ -- root of unity and $0\leq i< M$.
    We note that there is no ambiguity in the function defined by a Puiseux series, as the function $t^{1/M}$ refers locally to a unique branch of the $M^{\mathrm{th}}$ -- root function, and the other branches are given as conjugates by $\zeta_M^i$. Specifically, for $w$ a nonzero complex number written as $w=(r, \psi)$ in polar form, where $r\in \mathbb{R}_{>0}$ and $0\leq \psi <2\pi$, we have $w^{1/M}=(r^{1/M}, \psi/M)$, corresponding to the principal branch.
\end{remark}
So, for each $z\in \mathcal{Z}$, we use Theorem~\ref{thm:ntpu} to write $\bm{\tau}$ as a Puiseux series in $t-z$, after making a  choice of the series expansion to use. Essentially, this identifies $\bm{\tau}$ with a root of $\mu(x)$ over $\overline{K}\langle \langle t-z\rangle \rangle$.

\hfill

As stated earlier, this choice of embedding $\mathbf{K}\hookrightarrow \overline{K}\langle \langle t-z\rangle \rangle$ determines completely the cospecialisation map $\phi_{z}: \mathcal{F}_{z}\hookrightarrow \mathcal{F}_{\overline{\eta}}$. Following work of Igusa (Theorem~\ref{thm:igred}) we know that the elements of $\mathcal{F}_{z}$ can be identified as those solutions of  the $\ell$ -- torsion ideal $\prescript{(\ell)}{}{} \mathcal{I}_{\overline{\eta}}$ of $\mathrm{Pic}^{0}(\mathcal{X}_{\overline{\eta}})$ as a zero-dimensional ideal over $\overline{K}(t)$, which are in fact rational over $\overline{K}((t-z))$. The other elements of $\mathcal{F}_{\overline{\eta}}$ can be represented using rational function expressions in $\bm{\tau}$, which has, in turn, been identified with the Puiseux series $\lambda_z$ using our embedding. We sum up our efforts in Algorithm~\ref{algo:cospec}.

\begin{algorithm}
    \caption{\texttt{Computing a cospecialisation map at a singular point}}
    \label{algo:cospec}
     \begin{itemize}
        \item \textbf{Input:} A singular fibre $\mathcal{X}_{z}$ of the Lefschetz pencil $\pi:\mathcal{X}\rightarrow \mathbb{P}^{1}$ for a fixed $z\in \mathcal{Z}$.
        
        \item \textbf{Output:} The elements of $\mathrm{Pic}^{0}(\mathcal{X}_{\overline{\eta}})[\ell]$ represented as ${\overline{K(t)}}$ -- rational points in a projective space $\mathbb{P}^{M}$ using convergent Puiseux series around $z$.
    \end{itemize}

    \begin{algorithmic}[1]

    \STATE \label{algo:cospec1} Compute the $\ell$ -- division ideal $\prescript{(\ell)}{}{}{\mathcal{I}_{\overline{\eta}}}$  of $\mathrm{Pic}^{0}(\mathcal{X}_{\overline{\eta}})$ using Algorithm~\ref{algo:genelldiv}. 

    \STATE \label{algo:cospec2} Represent the $\ell^{2g}$ solutions of $\prescript{(\ell)}{}{}{\mathcal{I}_{\overline{\eta}}}$ over $\overline{K(t)}$ using a primitive element $\bm{\tau}$ and a zero-dimensional system solving algorithm such as \cite{rouillier}. In particular, an element $\gamma$ of $\mathrm{Pic}^{0}(\mathcal{X}_{\overline{\eta}})[\ell]$ is represented as a point in $\mathbb{P}^{M}$ with its coordinates being rational functions in $\bm{\tau}$ with coefficients from a $\mathrm{poly}(\ell)$ -- degree extension of $K$. 

    \STATE \label{algo:cospec3} Expand $\bm{\tau}$ as a Puiseux series $\lambda_z$ around $z$ using the algorithm from Theorem~\ref{thm:ntpu}, upto $\mathrm{poly}(\ell)$ precision.  Similarly rational functions in $\bm{\tau}$ also have convergent Puiseux series representations. This identifies each $\gamma$ uniquely by Lemma~\ref{lem:puibound}.

    \STATE \label{algo:cospec4} Return a representation of each $\gamma$ as a tuple $$[X_{0}^{(\gamma)}(t):\ldots : X_{M}^{(\gamma)}(t)],$$ where $X_{i}^{(\gamma)}(t)$ are Puiseux series in $t-z$.

\end{algorithmic}
\end{algorithm}

\begin{remark}
    By Theorem~\ref{thm:ntpu}, all the Puiseux expansions $X_i^{(\gamma)}(t)$ converge for all $t$ in a neighbourhood of $z$. In other words, they all converge for $\vert t-z\vert < \varepsilon_z$, where $\varepsilon_z\in \mathbb{R}_{>0}$ is the minimum of the radii of convergence of all the $X_{i}^{(\gamma)}(t)$. 
\end{remark}

\begin{lemma}
\label{lem:puibound}
    It suffices to specify
$$
\mathrm{poly}(\ell)
$$
coefficients of the Puiseux expansion of each $\gamma\in \mathcal{F}_{\overline{\eta}}$ around $z\in \mathcal{Z}$, in order to identify it uniquely. Further, the Weil height of each coefficient is bounded by a polynomial in $\ell$.
\end{lemma}
\begin{proof}
    The first statement follows from \cite[pg 3]{walsh}.( See also \cite[Theorem 4.5]{hilliker}). The bound for the height of the coefficients is provided by \cite[Theorem 1]{walsh}.
    
\end{proof}

\begin{remark}
    We `store' an algebraic number $\alpha$, by a pair consisting of its minimal polynomial and a floating-point approximation, to distinguish $\alpha$ from its conjugates. 
\end{remark}
We next note the following.

\begin{lemma}[Radius of convergence]
    \label{lem:radconv}
  There exists a polynomial $\Psi(x)\in \mathbb{Z}[x]$, with coefficients and degree independent of $\ell$, such that the common radius of convergence $\varepsilon_z$ satisfies

  $$
  \varepsilon_{z}>\frac{1}{\mathrm{exp}\left(\Psi(\ell)\right)}.
  $$
  
\end{lemma}

\begin{proof}
    Denote by $$\left(X^{(\gamma)}_{i}(t)\right)_{\gamma \in \mathcal{F}_{\overline{\eta}}}$$ the system of Puiseux expansions one obtains for the elements of $\mathcal{F}_{\overline{\eta}}$ around $z$. In particular, they are Laurent series in $\bm{t}=(t-z)^{1/M}$ for some $M$ bounded by a polynomial in $\ell$. Write 
    $$
    X^{(\gamma)}_{i}(t)=\sum_{j}\alpha^{(\gamma)}_{i, j}\bm{t}^{j}.
    $$
It converges on a disc $\vert \bm{t} \vert < \varepsilon_z $    where 

$$
\frac{1}{\varepsilon_z}=\limsup_{j\rightarrow \infty}{\vert \alpha^{(\gamma)}_{i, j} \vert^{\frac{1}{j}}}.
$$

Applying \cite[Corollary 4.6]{hickel} \footnote{see also Theorem 2.3 of loc. cit.}, we see that
$$
\vert \alpha^{(\gamma)}_{i, j} \vert \leq \mathrm{exp}\left( \Psi(\ell)\cdot j \right),
$$
where $\Psi(x)$ is a polynomial with coefficients and degree independent of $j$ and $\ell$. Taking the limit gives the result.

\end{proof}

\subsection{Specialisation to a smooth fibre}
\label{subsec:spec}

Consider the setup of Section~\ref{subsec:cospec}. Let $z\in \mathcal{Z}$. In this subsection, we indicate how we may specialise elements of $\mathcal{F}_{\overline{\eta}}$ realised as Puiseux expansions around $z$ using Algorithm~\ref{algo:cospec}, to elements of $\mathrm{Pic}^{0}(\mathcal{X}_{u_{z}})[\ell]$ for a `nearby' smooth fibre $\mathcal{X}_{u_{z}}$. We recall the following.
\begin{lemma}
\label{lem:isoco}
    Let $u\in \mathcal{U}$. Then, any cospecialisation map
    $$
    \phi_{u}: \mathcal{F}_{u}\rightarrow \mathcal{F}_{\overline{\eta}}
    $$
    is an isomorphism. Its inverse $\phi_{u}^{-1}$ associates a divisor in $\mathcal{F}_{\overline{\eta}}$ to the intersection with $\mathcal{X}_{u}$ of its closure in $\mathcal{X}$.
\end{lemma}
\begin{proof}
  The first statement follows from the fact that $\mathcal{F}\vert_{U}$ is a locally constant sheaf on $U$. See \cite{milne80} for more details.
\end{proof}

Now, consider again the splitting field $\mathbf{K}$ of $\prescript{(\ell)}{}{}\mathcal{I}_{\overline{\eta}}$. Under the natural embedding $\overline{K}(t)\hookrightarrow \overline{K}((t-u))$, we know that the elements of $\mathrm{Pic}^{0}(\mathcal{X}_{\overline{\eta}})[\ell]$ are rational over $\overline{K}((t-u))$ as the $\ell$ -- torsion of the generic fibre is unramified at $u$. We observe the following next.

\begin{lemma}
\label{lem:pairpres}
    Any specialisation $\phi_{u}^{-1}$ preserves the Weil pairing, i.e., for any $\gamma_1, \gamma_2 \in \mathcal{F}_{\overline{\eta}}$, we have
    $$
    \langle \gamma_{1}, \gamma_{2} \rangle = \langle \phi_{u}^{-1}(\gamma_1), \phi_{u}^{-1}(\gamma_2) \rangle,
    $$
    where the pairing on the left is the Weil pairing on $\mathrm{Pic}^{0}(\mathcal{X}_{\overline{\eta}})[\ell]$ and the one on the right is the Weil pairing on $\mathrm{Pic}^{0}(\mathcal{X}_{u})[\ell]$.
\end{lemma}
\begin{proof}
    Clear from the definition of specialisation.
\end{proof}

\begin{lemma}
    \label{lem:puiconv}
    Let $\gamma\in \mathcal{F}_{\overline{\eta}}$, and assume we have computed
    $$
    \gamma=[X_{0}^{(\gamma)}(t):\ldots :X_{M}^{(\gamma)}(t)]
    $$
    as a tuple of Puiseux series around $z\in \mathcal{Z}$ (truncated upto $\mathrm{poly}(\ell)$ coefficients so that any two $\gamma_{1}\neq \gamma_2$ in $\mathcal{F}_{\overline{\eta}}$ can be distinguished), with respect to the cospecialisation $\phi_{z}$. Then, for any $u_z\in \mathcal{U}$ with $\vert z-u_z\vert < \varepsilon_z/2$, the tuple representing $\gamma$ converges at $u_z$ to a specialisation  $\phi^{-1}_{u_z}(\gamma)\in \mathrm{Pic}^{0}(\mathcal{X}_{u_z})[\ell]$ of $\gamma$ at $u_z$.
\end{lemma}

\begin{proof}
It follows from the convergence properties of the associated Puiseux series (see \cite[2.2]{wall} for more details) that at $u_{z}$, $\gamma$ converges to a root of the zero-dimensional ideal $\prescript{(\ell)}{}{}\mathcal{I}_{u_{z}}$, or in other words, an $\ell$-torsion point $\gamma_{u_{z}}\in \mathrm{Pic}^{0}(\mathcal{X}_{u_{z}})[\ell]$. Now, as $u_{z}$ is a smooth specialisation for the ideal $\prescript{(\ell)}{}{}\mathcal{I}_{\overline{\eta}}$, we may, uniquely Hensel-lift this point $\gamma_{u_{z}}$ to a set of expansions $$\phi_{u_z}(\gamma_{u_{z}})=[Y_{0}(t):\ldots Y_{M}(t)]$$ where $Y_{i}(t)\in \overline{K}((t-u_z))$ converge in neighbourhood $W$ of $u_z$. The uniqueness of the lift of $\gamma_{u_{z}}$ implies that the tuples $[X^{(\gamma)}_{i}(t)]$ and $[Y_{i}(t)]$ represent the same analytic germs \footnote{being  solutions of $\prescript{(\ell)}{}{}\mathcal{I}_{\overline{\eta}}$, which are all distinct and $\ell^{2g}$ in number} on $W\cap \{u\in \mathbb{C} \ \vert \ \vert z-u\vert <\varepsilon_z/2\}$. This proves the claim.

\end{proof}

\begin{remark}
    Having fixed a cospecialisation $\phi_z$ at $z$, one automatically determines cospecialisation morphsisms $\phi_u$ for all $u$ in a neighbourhood of $z$ via the above lemma. We call these \textit{analytically compatible} cospecialisations.
    \end{remark}

We intend to use the above lemma to make the specialisation explicit. It remains to prove $\mathrm{poly}(\ell)$ -- bounds to separate roots of $\prescript{(\ell)}{}{}\mathcal{I}_{u_{z}}$ and derive the level of precision to determine which root it is that the associated expansions of $\gamma$ converge to. We deal with the first item initially, using a classical result from diophantine approximation.

\begin{lemma}
\label{lem:algbound}
  Let $\upsilon_1$ and $\upsilon_2$ be algebraic numbers occurring as roots of a polynomial $f(x)\in K[x]$ of degree $\mathbf{d}$ and height $\mathbf{h}$. Then
    $$
    \vert \upsilon_1 -\upsilon_2 \vert \geq \Gamma(\mathbf{d}, \mathbf{h}):= \frac{\sqrt{3}}{(\mathbf{d}+1)^{(2\mathbf{d}+1)/2}\cdot \mathbf{h}^{\mathbf{d}-1}}.
    $$
\end{lemma}

\begin{proof}
    See \cite[Corollary A.2]{bugeaudbook}.
\end{proof}

 In our context, $\mathbf{h}$ and $\mathbf{d}$ are both bounded by polynomials in $\ell$. This is because for a smooth $u\in \mathcal{U}$ of bounded height, the $\ell$ -- division system $\prescript{(\ell)}{}{}\mathcal{I}_{u}$ associated to $\mathrm{Pic}^{0}(\mathcal{X}_{u})$ has degree polynomial in $\ell$, and the algebraic numbers occurring as coefficients also have height bounded by a polynomial in $\ell$ (by Theorem~\ref{thm:tors_height}).
 Hence, we may write $$\Gamma(\ell):=\frac{1}{\mathrm{exp}(\Phi(\ell))}\leq \Gamma(\mathbf{d}, \mathbf{h})$$ where $\Phi(x)\in \mathbb{Z}[x]$ is a polynomial with coefficients and degree independent of $\ell$.

\begin{lemma}[Convergence-testing]
\label{lem:convbound}
    Let $\Lambda_1(t)=\sum_{j}\alpha_{j}t^{j/\ell}$ be an algebraic Puiseux series in $t$ occurring in a tuple representing $\gamma\in \mathcal{F}_{\overline{\eta}}$ in the context of Lemma~\ref{lem:puiconv}, around $z=0$ wlog. Write $\Lambda_2(t)=\sum_{j}\zeta_{\ell}^{j}\alpha_{j}t^{j/\ell}$ for its conjugate and let $u$ be an algebraic number of height bounded by a polynomial in $\ell$, with 
    $$\vert u\vert^{1/\ell}<\frac{1}{2\cdot\mathrm{exp}((\Psi(\ell))}$$ 
    
    such that both $\Lambda_1(t)$ and $\Lambda_2(t)$ converge at  $u$ to distinct, conjugate algebraic numbers $\upsilon_1$ and $\upsilon_2$ respectively. Then, it requires at most $\mathrm{poly}(\ell)$ precision to distinguish $\upsilon_1$ from $\upsilon_2$, i.e., to determine which series converges to which number.
\end{lemma}

\begin{proof}
    Write $\mathbf{t}:=t^{1/\ell}$, so we regard $\Lambda$ and $\Lambda'$ as power series in $\mathbf{t}$. We show firstly, that with $\mathrm{poly}(\ell)$ terms, we can approximate $\Lambda$ and $\Lambda'$ at $u$ to within $\Gamma(\ell)/4$ of $\upsilon_1$ and $\upsilon_2$ respectively. Denote by $\lambda^{(m)}_{1}(\mathbf{t})$ and $\lambda^{(m)}_{2}(\mathbf{t})$ the $m^{\mathrm{th}}$ partial sums of $\Lambda_1(\mathbf{t})$ and $\Lambda_2(\mathbf{t})$ respectively. Then, applying Lemma~\ref{lem:radconv}
    $$
    \vert\Lambda_{1}(u)-\lambda_{1}^{(m)}(u)\vert=\sum_{j>m}\vert\alpha_{j} \vert \cdot (\vert u \vert ^{1/\ell})^{j}\leq \sum_{j>m}(\mathrm{exp}(\Psi(\ell))\cdot u)^{j}\leq \sum_{j>m}\frac{1}{2^{j}},
    $$
    which can clearly be made less than $\Gamma(\ell)/4$ for a value of $m$ polynomial in $\ell$. So, we have
    $$
    \vert \upsilon_{1}-\lambda_{1}^{(m)}(u) \vert < \Gamma(\ell)/4 \  \ \mathrm{and} \ \ \vert \upsilon_{2}-\lambda_{2}^{(m)}(u) \vert < \Gamma(\ell)/4
    $$
for $m\in \mathbb{Z}_{>0}$ bounded by a polynomial in $\ell$. By Lemma~\ref{lem:algbound}, these truncations specify $\upsilon_1$ and $\upsilon_2$ uniquely and unambiguously as $\vert \upsilon_{1}- \upsilon_2 \vert > \Gamma(\ell)$.

\end{proof}

Combining Lemmas~\ref{lem:puiconv}, \ref{lem:algbound} and \ref{lem:convbound}, we have shown the following.

\begin{theorem}[Approximation]
\label{thm:approx}
    Let $\gamma \in \mathcal{F}_{\overline{\eta}}$ and let $z\in \mathcal{Z}$. Assume we have computed $\gamma$ as a tuple $[X_{0}^{(\gamma)}:\ldots : X_{M}^{(\gamma)}(t)]$ of Puiseux expansions truncated upto $\mathrm{poly}(\ell)$ coefficients, with respect to the cospecialisation $\phi_z$. Then, for $u_{z}$ of height bounded by $\mathrm{poly}(\ell)$ such that $\vert z- u_{z}\vert < \varepsilon_{z}/2$, it is possible to determine with $$\mathrm{poly}(\ell) \ \text{space, time and precision complexity,}$$ the unique analytically compatible specialisation $\gamma_{u_{z}}=\phi_{u_{z}}^{-1}(\gamma)$ as the tuple $[x_{0}:\ldots : x_{M}]$ that  $[X^{(\gamma)}_{0}(t):\ldots : X_{M}^{(\gamma)}(t)]$ converges to at $u_{z}$.
\end{theorem}
\qed

The next task is to make the specialisation map explicit. Let $z\in \mathcal{Z}$. In Algorithm~\ref{algo:cospec}, we obtained a representation of $\mathcal{F}_{\overline{\eta}}$ as Puiseux series around $z$, with the common minimal radius of convergence $\varepsilon_z$. In Algorithm~\ref{algo:recenter}, we indicate how to compute, for $\gamma\in \mathcal{F}_{\overline{\eta}}$ obtained via Puiseux series expansions around $z$; the specialisation $\phi_{u_{z}}^{-1}(\gamma)\in \mathrm{Pic}^{0}(\mathcal{X}_{u_{z}})[\ell]$ for $u_{z}\in \mathcal{U}$ such that $\vert z-u_{z}\vert < \varepsilon_z$.

\begin{algorithm}
    \caption{Re-centering}
    \label{algo:recenter}
     \begin{itemize}
        \item \textbf{Input:} An element $\gamma\in \mathcal{F}_{\overline{\eta}}$ represented by a tuple $[X_{0}^{\gamma}(t):\ldots :X_{M}^{(\gamma)}(t)]$ of Puiseux series around $z$ as a $\mathbf{K}$ -- rational point in $\mathbb{P}^{M}$ (via Algorithm~\ref{algo:cospec}), and a smooth point $u\in \mathcal{U}$ with $\vert u- z\vert <\varepsilon_z$.
        
        \item \textbf{Output:} The specialisation $\phi_{u_z}^{-1}(\gamma)\in \mathrm{Pic}^{0}(\mathcal{X}_{u_z})[\ell]$. 
        
    \end{itemize}

    \begin{algorithmic}[1]

    \STATE \label{algo:recenter1} Specialise the ideal $\prescript{(\ell)}{}{}\mathcal{I}_{\overline{\eta}}$ at $u_z$ to obtain the $\ell$ -- division ideal $\prescript{(\ell)}{}{}\mathcal{I}_{u_{z}}$ for $\mathrm{Pic}^{0}(\mathcal{X}_z)$ by Section~\ref{app:igusa}.

    \STATE \label{algo:recenter2}  Compute the $\ell^{2g}$ distinct $\ell$ -- torsion elements $\mathrm{Pic}^{0}(\mathcal{X}_{u_z})[\ell]$ via a zero-dimensional system solving algorithm (\cite{rouillier}) applied to $\prescript{(\ell)}{}{}\mathcal{I}_{u_z}$.

    \STATE \label{algo:recenter3}  The input tuple $[X_{0}^{(\gamma)}(t):\ldots:X^{(\gamma)}_{M}(t)]$ actually converges at $u_z$ to a point $[x_{0}:\ldots:x_{M}]\in \mathrm{Pic}^{0}(\mathcal{X}_{u_{z}})$. Determine the point as a tuple of algebraic numbers by using Theorem~\ref{thm:approx} and matching with the points computed in Step 2.
    
\end{algorithmic}
\end{algorithm}

\subsection{Computing vanishing cycles and monodromy}
\label{subsec:computevan}

The goal of this subsection is to compute the monodromy action on $\mathcal{F}_{\overline{\eta}}$. Additionally, we also compute the local monodromy at each singular point, explicitly computing each vanishing cycle in the process. This algebraic computation of monodromy can be understood as an algebraic, finite coefficient analogue of the work \cite{lpv} extended to the case of a Lefschetz pencil on an arbitrary smooth projective surface (as opposed to a hypersurface). 

\begin{remark}
    The vanishing cycle $\delta_{z}$ depends on the chosen cospecialisation $\phi_{z}:\mathcal{F}_{z}\hookrightarrow \mathcal{F}_{\overline{\eta}}$. Hence, it would be more accurate to write $\phi_{z}(\delta_{z})\in \mathcal{F}_{\overline{\eta}}$ for the vanishing cycle, but we abuse notation by referring to it as just $\delta_{z}$. This is because the cospecialisations $\phi_{z}$ have already been chosen or determined, as will be seen below.
\end{remark}

As stated in Section~\ref{subsec:cospec}, for $z\in \mathcal{Z}$, the vanishing cycle $\delta_{z}\in \mathcal{F}_{\overline{\eta}}$ is determined uniquely upto sign by the Picard-Lefschetz formulas after picking a $\overline{K}(t)$ -- embedding $\mathbf{K}\hookrightarrow \overline{K}\langle \langle t-z \rangle \rangle$. Firstly, write $Z=\{z_{1}, \ldots , z_{r}\}$ as an ordered set of distinct points for $r\in \mathbb{Z}_{>0}$. We make certain preliminary simplifications following the discussion before \cite[Theorem 3.23]{milne80}.

\hfill

Choose $\zeta_{s}:=\mathrm{exp}(2\pi i/s)$ as a generator of $\mu_{s}(\overline{K})$ for each $s$ so that $\zeta_{l}=\zeta_{sl}^{s}$. Let $I^{\mathfrak{t}}_{z_{j}}$ denote the tame inertia group at $z_{j}$ and let $\sigma_{j}$ be its generator. We need to choose embeddings $I^{\mathfrak{t}}_{z_{j}}\hookrightarrow \mathrm{Gal}(\overline{K(t)}/\overline{K}(t))$ in such a way that the $\sigma_{j}$ together generate the tame fundamental group $\pi_{1}(U, \overline{\eta})$ and $\prod_{j=1}^{r}\sigma_{j}=1$. This implies that we are freely permitted to choose the embeddings for $1\leq j\leq r-1$ but the embedding for $j=r$ is decided by the others, so that $$\sigma_{r}=\prod_{j=1}^{r-1}\sigma_{r-j}^{-1}\in \pi_{1}^{\mathfrak{t}}(U, \overline{\eta}).$$
Further, for all $1\leq j \leq r$, the canonical generator $\sigma_{j}$ of the inertia $I^{\mathfrak{t}}_{z_{j}}$ acts as 
$$
\sigma_{j}\left(t-z_{j}\right)^{1/s}=\zeta_{s}\left(t-z_{j}\right)^{1/s}.
$$
What this means for us, is that the cospecialisation maps $\phi_{{z}_{j}}:\mathcal{F}_{z_{j}}\hookrightarrow \mathcal{F}_{\overline{\eta}}$ are determined by arbitrary embeddings for $1\leq j \leq r-1$, but once these choices have been made, the {last cospecialisation} $\phi_{z_{r}}:\mathcal{F}_{z_{r}}\hookrightarrow \mathcal{F}_{\overline{\eta}}$ is completely determined by the previously made choices. With these simplifications, the Picard-Lefschetz formula (\ref{eqn:piclef}) becomes
\begin{equation}
    \label{eqn:simpiclef}
    \sigma_{j}(\gamma)=\gamma - \langle \gamma, \delta_{z_{j}}\rangle \delta_{z_{j}}
\end{equation}
for $\gamma \in \mathcal{F}_{\overline{\eta}}$ and $1\leq j \leq r$. We now give a method, such that given $z_{j}\in \mathcal{Z}$ for $1\leq j\leq r-1$, and $u_{j}\in \mathcal{U}$ with $\vert z_{j}- u_{j} \vert<\varepsilon_{z_j}$, we compute $\phi_{u_j}^{-1}(\delta_{z_{j}})$ as an element of $\mathrm{Pic}^{0}(\mathcal{X}_{u_{j}})[\ell]$.

\begin{algorithm}
\caption{\texttt{Computing vanishing cycles}}
\label{algo:van}
\begin{itemize}
    \item \textbf{Input:} A singular point $z\in \mathcal{Z}\setminus \{z_r\}$ and a smooth point $u_z$ such that $\vert z-u_z \vert < \varepsilon_z$.
    \item \textbf{Output:} An element $\delta_{z}\in \mathcal{F}_{\overline{\eta}}$ unique upto sign, that is the vanishing cycle at $z$ with respect to the cospecialisation $\phi_{z}$ of Algorithm~\ref{algo:cospec}.
    

    \end{itemize}

    \begin{algorithmic}[1]
\STATE \label{algo:van1} Obtain a representation of $\mathcal{F}_{\overline{\eta}}$ as Puiseux series around $z$ using Algorithm~\ref{algo:cospec}.

\STATE \label{algo:van2}  Choose $\gamma=[X_{0}^{(\gamma)}(t):\ldots:X_{M}^{(\gamma)}(t)]\in \mathcal{F}_{\overline{\eta}}\setminus \phi_{z}(\mathcal{F}_{z})$. This reduces to choosing a $\gamma$ for which at least one of the Puiseux series $X_{j}^{(\gamma)}(t)$ is ramified at $z$, i.e., is a true Puiseux series and not in fact a Laurent series.

\STATE \label{algo:van3}  Writing
$$
X_i^{(\gamma)}(t)=\sum_{j}\alpha^{(\gamma)}_{i, j}(t-z)^{j/\ell}
$$
evaluate $$\sigma_{z}(\gamma)=[X_{0}^{(\sigma_{z}(\gamma))}(t): \ldots : X_{M}^{(\sigma_{z}(\gamma))}(t)]$$
where
$$
X_i^{(\sigma_{z}(\gamma))}(t)=\sum_{j}\alpha^{(\gamma)}_{i, j}\zeta_{\ell}^{j}(t-z)^{j/\ell}.
$$

\STATE \label{algo:van4}  Compute the element $\phi_{u_{z}}^{-1}(\sigma_{z}(\gamma))\in \mathrm{Pic}^{0}(\mathcal{X}_{u_{z}})[\ell]$ using the specialisation of Algorithm~\ref{algo:recenter}.

\STATE \label{algo:van5} Compute $\phi_{u_{z}}^{-1}(\gamma)$ using Algorithm~\ref{algo:recenter}.

\STATE \label{algo:van6}  Compute $$\delta:=\phi_{u_{z}}^{-1}(\sigma_{z}(\gamma))-\phi_{u_{z}}^{-1}(\gamma)$$ using the explicit group law on $\mathrm{Pic}^{0}(\mathcal{X}_{u_{z}})$ (using Theorem~\ref{thm:eqnjac}).

\STATE \label{algo:van7} Use the inverse of the abstract Abel map of Section~\ref{app:jac} (Algorithm~\ref{algo:abel}) to represent the $\ell$ -- torsion points $\phi_{u_{z}}^{-1}(\gamma)$ and $\delta$ as divisors on $\mathcal{X}_{u_{z}}$.

\STATE \label{algo:van8}  Use the divisorial representation in Step 7 to compute the Weil pairing $$a:=\langle \phi_{u_{z}}^{-1}(\gamma), \delta\rangle \in \mathbb{Z}/\ell \mathbb{Z}$$ on $\mathrm{Pic}^{0}(\mathcal{X}_{u_{z}})[\ell]$ using Algorithm~\ref{algo:pairing}.

\STATE \label{algo:van9}  Applying (\ref{eqn:piclefcomp}), compute
$$
\phi_{u_{z}}^{-1}(\delta_z)=\pm (\sqrt{-a^{-1}}) \cdot \delta\in \mathrm{Pic}^{0}(\mathcal{X}_{u_{z}})[\ell]
$$
via the explicit addition law (Theorem~\ref{thm:eqnjac}), and make an arbitrary choice of sign.

\STATE \label{algo:van10}  With knowledge of $\phi_{u_{z}}^{-1}(\delta_z)$, identify it with the correct tuple of Puiseux expansions around $z$  and return $\delta_z$ as a rational function in the primitive element $\bm{\tau}$.

    \end{algorithmic}
    
\end{algorithm}

\begin{theorem}
\label{thm:algovan}
      Algorithm~\ref{algo:van} uniquely determines the vanishing cycle at each $z\in \mathcal{Z}\setminus \{z_r\}$, upto sign.
\end{theorem}
\begin{proof}
Let $\gamma\in \mathcal{F}_{\overline{\eta}}\setminus \phi_{z}(\mathcal{F}_{z})$. By Section~\ref{subsec:cospec}, we know that after a choice of embedding, we may write $$\gamma=[X^{(\gamma)}_{0}(t):\ldots : X^{(\gamma)}_{M}(t)]$$ as a tuple of Puiseux series around $z$, representing a $\overline{K(t)}$ -- rational point of $\mathrm{Pic}^{0}(\mathcal{X}_{\overline{\eta}})$.
By Theorem~\ref{thm:igred}, we know that the image  $\phi_{z}(\mathcal{F}_{z})$ is all rational over $\overline{K}((t-z))$, so in order to choose $\gamma$ from outside $\mathcal{F}_{z}$, it suffices to ensure one associated Puiseux expansion ramifies at $z$. \\ \\
Having chosen compatible generators $\zeta_{s}$ for $\mu_{s}(\overline{K})$, we may identify the inertia $I^{\mathfrak{t}}_{z}$ at $z$ as
$$
I^{\mathfrak{t}}_{z}\simeq \prod_{\ell' \ \mathrm{prime}} \mathbb{Z}_{\ell'}.
$$
Our choice of topological generator $\sigma_{z}$ sends $(t-z)^{1/\ell}$ to $\zeta_{\ell}(t-z)^{1/\ell}$, and acts termwise on the Puiseux expansions associated to $\gamma$. In this way, the action of $\sigma_{z}$ is realised as an automorphism of $\mathcal{F}_{\overline{\eta}}$, that precisely fixes $\phi_{z}(\mathcal{F}_{z})$. In particular, since $\gamma\not\in \phi_{z}(\mathcal{F}_{z })$, we have $\sigma_{z}(\gamma)\ne \gamma$. Therefore, by the Picard-Lefschetz formula (\ref{eqn:simpiclef}), we know $\langle \gamma, \delta_{z} \rangle \neq 0$. \\ \\ For a $u_{z}$ such that $\vert z-u_{z} \vert < \varepsilon_{z}$, we know that the Puiseux series $X^{(\gamma)}_{i}(t)$ all converge at $t=u_{z}$. Further, by Section~\ref{subsec:spec}, Algorithm~\ref{algo:recenter} computes the unique (and distinct) specialisations $\phi_{u_{z}}^{-1}(\sigma_{z}(\gamma))$ and $\phi_{u_{z}}^{-1}(\gamma)$ of $\gamma$ to the $\ell$ -- torsion of $\mathrm{Pic}(\mathcal{X}_{u_{z}})$. Set $$
\delta:=\phi_{u_{z}}^{-1}(\sigma_{z}(\gamma))-\phi_{u_{z}}^{-1}(\gamma)=\phi_{u_{z}}^{-1}(\sigma_{z}(\gamma)-\gamma),
$$
and $a:=\langle \phi_{u_{z}}^{-1}(\gamma), \delta \rangle$. Note that a priori, $a\in \mu_{\ell}(\overline{K})$, but we have then taken its discrete logarithm with respect to the generator $\zeta_{\ell}$. It remains to show the following.

\begin{lemma}
\label{lem:piclef}
The vanishing cycle $\delta_{z}$ at $z$ can be computed as
\begin{equation}
\label{eqn:piclefcomp}
\delta_{z}= \pm \phi_{u_{z}}\left((\sqrt{-a^{-1}}) \cdot \delta\right)
\end{equation}
\end{lemma}
\begin{proof}
   First, we see that $a\neq 0$ as an element of $\mathbb{Z}/\ell \mathbb{Z}$. Indeed, $$a=\langle \phi^{-1}_{u_{z}}(\gamma), \delta\rangle =\langle \phi_{u_{z}}^{-1}(\gamma), \phi^{-1}_{u_{z}}\left(\sigma_{z}(\gamma)-\gamma\right)\rangle=\langle \gamma, \sigma_{z}(\gamma)-\gamma\rangle=\langle \gamma, \sigma_{z}(\gamma) \rangle \ne 0.$$ 
Further, we know by the Picard-Lefschetz formulas, or Section~\ref{app:igusa}, Theorem~\ref{thm:igg} that $\phi_{u_{z}}(\delta)=\sigma_{z}(\gamma)-\gamma \ \in \ <\delta_{z}> \ \subset \mathcal{F}_{\overline{\eta}}$. Therefore, writing $$c\cdot \phi_{u_{z}}(\delta)= \delta_{z}$$ for some $c\in (\mathbb{Z}/\ell \mathbb{Z})^{*}$, we see $$
\sigma_{z}(\gamma)-\gamma=-\langle \gamma, \delta_{z} \rangle \delta_{z}=-c\cdot \left(\langle \gamma, c\cdot \phi_{u_z}(\delta)\rangle\right)\cdot \phi_{u_z}(\delta)= -c^{2}\cdot \left(\langle \gamma, \phi_{u_z}(\delta) \rangle \right)\cdot \phi_{u_z}(\delta) = \phi_{u_z}(\delta). 
$$
Equating coefficients, we have $$
a=\langle \phi_{u_{z}}^{-1}, \delta \rangle=\langle \gamma, \phi_{u_z}(\delta) \rangle =-c^{-2}.
$$
Therefore, we see 
$$
c=\pm\sqrt{-a^{-1}}.
$$

\end{proof}

Thus, the specialised vanishing cycle $\phi^{-1}_{u_{z}}(\delta_{z})\in \mathrm{Pic}^{0}(\mathcal{X}_{u_{z}})[\ell]$ is computed. This completes the proof of Theorem~\ref{thm:algovan}.

\end{proof}
\begin{remark}
    We check that $-a$ is indeed a square in $\mathbb{Z}/\ell \mathbb{Z}$ as $$-a=-\langle \gamma, \phi_{u_z}(\delta)\rangle=-\langle \gamma, \sigma_{z}(\gamma)\rangle=-\langle \gamma, -(\langle \gamma, \delta_{z}\rangle )\cdot \delta_{z}\rangle=(\langle \gamma, \delta_{z}\rangle)^{2}.$$
\end{remark}
    We emphasise again that the cospecialisations $\phi_{z_j}: \mathcal{F}_{z_{j}}\rightarrow \mathcal{F}_{\overline{\eta}}$ have only been made explicit for $1\leq j\leq r-1$, as arbitrary choices were allowed for the associated embeddings $I^{\mathfrak{t}}_{z_{j}}\hookrightarrow \mathrm{Gal}\left(\overline{K(t)}/\overline{K}(t)\right)$. However, the final embedding $I^{\mathfrak{t}}_{z_{r}}\hookrightarrow \mathrm{Gal}\left(\overline{K(t)}/\overline{K}(t)\right)$ is completely determined by the previous ones, via the  relation $\prod_{j=1}^{r}\sigma_{j}=1$ in $\pi_{1}^{\mathfrak{t}}(U, \overline{\eta})$. Hence, an explicit representation of the \textit{last vanishing cycle} $\delta_{z_r}$ can be computed by just using the knowledge of the action of the other inertia generators. Thus, this enables us to compute the subspace $\mathcal{E}_{\overline{\eta}}\subset \mathcal{F}_{\overline{\eta}}$ of vanishing cycles. We sum up, with an algorithm computing the action of the generators $\sigma_{j}$ for $1\leq j< r$, of the geometric monodromy.

\begin{algorithm}
    \caption{\texttt{Computing the monodromy}}
    \label{algo:mono}
    \begin{itemize}
    \item \textbf{Input:} An element $\gamma\in \mathcal{F}_{\overline{\eta}}$ presented as a tuple of rational functions in the primitive element $\bm{\tau}$.
    \item \textbf{Output:} For each $z_j\in \mathcal{Z}\setminus \{z_r\}$, the element $\sigma_j(\gamma)$, again presented as a tuple of rational functions in $\bm{\tau}$.

    \end{itemize}

\begin{algorithmic}[1]

\STATE \label{algo:mono1} For $z\in \mathcal{Z}\setminus \{z_r\}$, expand $\gamma$ as a Puiseux series around $z$ and compute $\sigma_{z}(\gamma)$ as in Step~\ref{algo:van3} of Algorithm~\ref{algo:van}.

\STATE \label{algo:mono2} Express $\sigma_z(\gamma)$, which is now represented as a tuple of Puiseux expansions around $z$, as a tuple of rational functions in $\bm{\tau}$, using the Puiseux expansion $\lambda_z$ for $\bm{\tau}$ and linear algebra.

\STATE \label{algo:mono3} Return the tuple of rational functions in $\bm{\tau}$.

\end{algorithmic}
\end{algorithm}
We conclude with a table drawing a parallel with monodromy computations in the complex analytic setting, such as \cite{lpv}.

\begin{table}[h!]
  \centering
 \renewcommand{\arraystretch}{1.8}
  \begin{tabular}{|p{8cm}|p{8cm}|p{8cm}|}
    \hline
    \textbf{Analytic side} & \textbf{\'Etale side} \\
    \hline
    $\pi_{1}^{\text{top}}(\mathcal{U}, u)$ & $\pi_{1}^{\text{\'et}}(\mathcal{U}, \overline{u})$ \\
    \hline
    Generator $\bm{\sigma}_j$ & Topological generator $\sigma_j$ \\
    \hline
    Loop based at $u$ going around a puncture $z$ & Embedding $I_z\hookrightarrow \mathrm{Gal}(\overline{K(t)}/K(t))$, together with isomorphism of fiber fuctors at $u$ and geometric generic point $\overline{\eta}=\mathrm{Spec}(\overline{K(t)})$. \\
    \hline
  \end{tabular}
  \caption{Analytic-\'etale comparison}
  \label{tab:comp}
\end{table}

\section{The Edixhoven subspace}
\label{sec:algos}
In this section, we describe how to compute the Galois action on the second \'etale cohomology. We begin with a high-level description of the strategy.
\begin{itemize}
    \item Having computed the monodromy, compute the normalisation of $\mathbb{P}^{1}$ in the function field of the \'etale cover $\mathcal{V}\rightarrow \mathcal{U}$ trivialising the locally constant sheaf $\mathcal{F}\vert_{\mathcal{U}}:=\mathrm{R}^{1}\pi_{\star}\mu_{\ell}\vert_{\mathcal{U}}$. Write $j:\mathcal{U}\rightarrow \mathbb{P}^{1}$ for the inclusion, and denote by $\mathcal{E}\subset \mathcal{F}\vert_{\mathcal{U}}$, the locally constant subsheaf of vanishing cycles.

    \item Let $\tilde{\mathcal{V}}\rightarrow \mathbb{P}^{1}$ now be the smooth curve so obtained, ramified at $\mathcal{Z}$. Then, the Galois action on $\mathrm{H}^{1}(\mathbb{P}^{1}, \mathcal{F})\subset \mathrm{H}^{2}(\mathcal{X}, \mu_\ell)$ can be computed from the action of Galois on the subspace of $\mathrm{H}^{1}(\tilde{\mathcal{V}},\mu_\ell)\otimes_{\mathbb{F}_{\ell}}\mathcal{E}_{\overline{\eta}}$, given by those tensors invariant under the diagonal action of $G$.
    \item The action of $G$ on $\tilde{\mathcal{V}}$ extends naturally to an action on $\mathrm{H}^{1}(\tilde{\mathcal{V}}, \mu_\ell)$. One then isolates the \textit{Edixhoven subspace} $\mathbb{E}\subset \mathrm{H}^{1}(\tilde{\mathcal{V}}, \mu_\ell)$, i.e., the subspace spanned by all copies of $\mathcal{E}_{\overline{\eta}}^{\vee}$ inside it, by working over a finite field, modulo a small auxiliary prime $\mathfrak{P}$ of good reduction, distinct from $\ell$.

    \item Calling $\tilde{\mathcal{V}}_{\mathfrak{P}}$ the curve obtained upon reduction, we obtain its zeta function by counting points, and isolate the Edixhoven subspace $\mathbb{E}_{\mathfrak{P}}$ (which is defined over a poly-bounded extension) with knowledge of the monodromy action.

    \item The subspace $\mathbb{E}_{\mathfrak{P}}$ is then lifted $\mathfrak{P}$ - adically, using Hensel's lemma, to the characteristic zero subspace $\mathbb{E}$ following Mascot \cite{mascothensel}, from which the Galois action is subsequently computed.

\end{itemize}
\subsection{The trivialising cover}
Consider the \'etale cover $\mathcal{V}\rightarrow \mathcal{U}$ that trivialises the locally constant sheaf $\mathcal{F}\vert_\mathcal{U}=R^{1}\pi_{\star}\mu_{\ell}\vert_{\mathcal{U}}$ (and hence, also $\mathcal{E}$), i.e., $\mathcal{F}\vert_{\mathcal{V}}=\mu_{\ell}^{\oplus 2g}$. One then  normalises the function field of $\mathbb{P}^{1}$ in the Galois closure of the field $\overline{K}(\mathrm{Jac}(\mathcal{X}_{\overline{\eta}})[\ell])$ that the relative $\ell$ -- torsion $\mathrm{Jac}(\mathcal{X}_{\overline{\eta}})[\ell]$ of the Jacobian of the generic fibre is defined over, to obtain $\tilde{\mathcal{V}}\rightarrow \mathbb{P}^{1}$. Passage to the Galois closure of a field is efficiently possible, simply by computing a primitive element, and going to its splitting field.

As seen earlier, this extension is of degree bounded by a polynomial in $\ell$, and a birational planar model of the curve representing this extension can be computed via a primitive element. A representation for $\tilde{\mathcal{V}}$ is computed via normalisation, for which there is a polynomial-time (in the genus $\mathfrak{g}$ of the curve) algorithm \cite{Kozen}. Further, the associated map $\mathfrak{j}:\mathcal{V}\rightarrow \mathcal{U}$ can be computed in polynomial-time. The map on the smooth compactifications $\tilde{\mathfrak{j}}:\tilde{\mathcal{V}}\rightarrow \mathbb{P}^{1}$ is ramified only at $\mathcal{Z}$, and its degree is bounded by a polynomial in $\ell$. \\ \\ Now, we assume that the prime $\ell$ is such that the integral $\ell$ -- adic cohomology groups of $\mathcal{X}$ are all torsion free. This is fine, as we are interested in the growing -- $\ell$ regime, and this condition is true for all $\ell$ larger than a function of the data of $\mathcal{X}$. 
\begin{theorem}
\label{thm:galiso}
    We have the following isomorphism of $\mathrm{Gal}(\overline{K}/K)$ -- modules
    \begin{equation}
        \label{eqn:impgal}
        \mathrm{H}^{1}(\mathbb{P}^{1}, \mathcal{F})\simeq \mathrm{H}^{1}(\mathbb{P}^{1}, j_{\star}\mathcal{E})\simeq \left(\mathrm{H}^{1}(\tilde{\mathcal{V}}, \mu_\ell)\otimes M\right)^{G}       
    \end{equation}
    where $M=\mathcal{E}_{\overline{\eta}}$.
    
\end{theorem}
\begin{proof}
 The first isomorphism follows from the fact that $\mathcal{F}=\mathrm{R}^{1}\pi_{\star}\mu_{\ell}\simeq j_{\star}j^{\star}\mathcal{F}\simeq j_{\star}{\mathcal{E}}\oplus \underline{\mathcal{A}}$, where $\underline{\mathcal{A}}$ is the constant sheaf associated to $\mathrm{H}^{1}(\mathcal{X}, \mu_{\ell})$. This is because, due to torsion-freeness, hard-Lefschetz holds modulo $\ell$; and the cohomology of a constant sheaf vanishes on $\mathbb{P}^{1}$. \\ \\
Next, consider the Hochschild-Serre spectral sequence \cite[Theorem 14.9]{milnelec}
    $$
    \mathrm{H}^{i}(G, \mathrm{H}^{j}(\mathcal{V}, \mathcal{E}\vert_{\mathcal{V}}))\Rightarrow \mathrm{H}^{i+j}(\mathcal{U}, \mathcal{E})
    $$
    associated to the Galois cover $\mathcal{V}\rightarrow \mathcal{U}$. One has the five-term long exact sequence
    $$
    0\rightarrow \mathrm{H}^{1}(G, M)\rightarrow \mathrm{H}^{1}(\mathcal{U}, \mathcal{E})\rightarrow \mathrm{H}^{1}(\mathcal{V}, \mathcal{E}\vert_{\mathcal{V}})^{G}\rightarrow \mathrm{H}^{2}(G, M)\rightarrow \mathrm{H}^{2}(\mathcal{U}, \mathcal{E}).
    $$

Now, as the integral $\ell$ - adic cohomology groups are torsion-free, we know by \cite[Theorem 13]{gcd}, that $G=\mathrm{Sp}(\mathcal{E}_{\overline{\eta}})=\mathrm{Sp}(M)$. In particular, we have that $\mathrm{H}^{i}(G, M)=0$ for $1\leq i \leq 2$, as the centre has order $2$ and acts non-trivially on $M$ (for $\ell>2$, which we assume anyway)\footnote{see \cite{jones}}. Therefore, we have $$\mathrm{H}^{1}(\mathcal{U}, \mathcal{E})\simeq \mathrm{H}^{1}({\mathcal{V}}, \mathcal{E}\vert_{\mathcal{V}})^{G}\simeq \left(\mathrm{H}^{1}(\mathcal{V}, \mu_{\ell})\otimes M\right)^{G}.$$ The passage to $\tilde{\mathcal{V}}$ follows from the fact that global cohomology classes in $\mathrm{H}^{1}(\mathbb{P}^{1}, j_{\star}\mathcal{E})$ extend over the punctures as well (by excision) as in the following diagram with rows exact

    \[
\begin{tikzcd}
0 \arrow{r}  & \mathrm{H}^{1}(\mathbb{P}^{1}, j_{\star}\mathcal{E}) \arrow{r}  & \mathrm{H}^{1}(\mathcal{U}, \mathcal{E}) \arrow{r} \arrow{d}{\simeq} & \mathrm{H}^{2}_{\mathcal{Z}}(\mathbb{P}^{1}, j_{\star}\mathcal{E}) \\
0 \arrow{r} & \mathrm{H}^{1}(\tilde{\mathcal{V}}, \tilde{\mathfrak{j}}^{\star}j_{\star}\mathcal{E})^{G}  \arrow{r} & \mathrm{H}^{1}(\mathcal{V}, \mathfrak{j}^{\star}\mathcal{E})^{G} \arrow{r} & \left(\mathrm{H}^{2}_{\tilde{\mathcal{Z}}}(\tilde{\mathcal{V}}, \tilde{\mathfrak{j}}^{\star}j_{\star}\mathcal{E})\right)^{G}
\end{tikzcd}
\]
where $\tilde{\mathcal{Z}}\subset \tilde{\mathcal{V}}$ is the finite set of points lying above $\mathcal{Z}$ under $\tilde{\mathfrak{j}}:\tilde{\mathcal{V}}\rightarrow \mathbb{P}^{1}$.

\end{proof}

 \begin{remark}
        The group $\mathrm{Gal}(\overline{K}/K)$ acts on $M=\mathcal{E}_{\overline{\eta}}$ cyclotomically, as the arithmetic \'etale fundamental group of $\tilde{\mathcal{V}}$ does. Taking Tate twists into account, the isomorphism boils down to $$\mathrm{H}^{1}(\mathbb{P}^{1}, \mathcal{F})\simeq \left(\mathrm{H}^{1}(\tilde{\mathcal{V}}, \mu_\ell)(-1)\otimes \mu_{\ell}^{\dim M}\right)^{G}$$ as $\mathrm{Gal}(\overline{K}/K)$ -- modules, with the diagonal action (c.f. \cite[Theorem 2.3]{mascot2023explicit}). 
    \end{remark}
\begin{definition}
    We define the \textit{Edixhoven subspace} $\mathbb{E}$ as
    $$
    \mathbb{E}:=\sum_{\phi \in \mathrm{Hom}_{G}\left(M^{\vee},  \mathrm{H}^{1}(\tilde{\mathcal{V}}, \mu_{\ell})\right)} \mathrm{im}(\phi)\subset \mathrm{H}^{1}(\tilde{\mathcal{V}}, \mu_{\ell}).
    $$
\end{definition}
The point of the definition is the following observation.
$$
\left(\mathrm{H}^{1}(\tilde{\mathcal{V}}, \mu_{\ell})\otimes M\right)^{G}\simeq \mathrm{Hom}_{G}\left(M^{\vee}, \mathrm{H}^{1}(\tilde{\mathcal{V}}, \mu_{\ell})\right)\simeq \mathrm{Hom}_{G}\left(M^{\vee}, \mathbb{E}\right)\simeq \left(\mathbb{E}\otimes M\right)^{G}
$$
where the first isomorphism is from standard tensor-hom and the second is because, by definition of $\mathbb{E}$, all $G$ - equivariant homomorphisms from $M^{\vee}$ to $\mathrm{H}^{1}(\tilde{\mathcal{V}}, \mu_{\ell})$ actually have image inside $\mathbb{E}$.

The algorithmic upshot is that $\dim \mathbb{E}$ is independent of $\ell$ (being bounded by $\beta_2\cdot \dim M$, where $\beta_2$ is the second Betti number of $\mathcal{X}$), whereas $\dim \mathrm{H}^{1}(\tilde{\mathcal{V}}, \mu_{\ell})=2\mathfrak{g}$, where $\mathfrak{g}$ depends polynomially on $\ell$.

\subsection{Geometric Galois action}
\label{subsec:geomgal}
 In this subsection, we describe how to compute the $G$-action on points of $\tilde{\mathcal{V}}$.
\begin{itemize}
    \item Consider the primitive element $\bm{\tau}$ for the field extension $\overline{K}(\tilde{\mathcal{V}})/\overline{K}(\mathbb{P}^{1})$.
    \item The extension has Galois group $G$, the geometric monodromy group. For each generator $\rho_{\ell}(\sigma_j)\in G$ for $1\leq j< r$, express $\sigma_j(\bm{\tau})$ as a rational function of $\bm{\tau}$, akin to Algorithm~\ref{algo:mono}.
    \item As each $\sigma_j$ gives rise to a birational automorphism of the smooth projective curve $\tilde{\mathcal{V}}$, it hence extends to an isomorphism, which can be given in terms of polynomials, using an efficient normalisation algorithm \cite{Kozen}.
    \item Hence simply evaluate the corresponding isomorphism on the input point, this gives the $G$ -- action on the points of $\tilde{\mathcal{V}}$.

    \item This extends to an action on $\mathrm{Jac}(\tilde{\mathcal{V}})$, via divisors.
\end{itemize}

\subsection{Isolating the Edixhoven subspace}
\label{subsec:arithgal}
We first give a method to isolate the Edixhoven subspace $\mathbb{E}\subset \mathrm{Jac}(\tilde{\mathcal{V}})[\ell]$  that is relevant for the Galois contribution on the second \'etale cohomology of the input surface. For this, we make use of an auxiliary prime $\mathfrak{P}$ of good reduction, distinct from $\ell$, and work with the positive-characteristic curve $\tilde{\mathcal{V}}_{\mathfrak{P}}$.

\begin{algorithm}
    \caption{\texttt{Computing the Edixhoven subspace modulo} $\mathfrak{P}$}
    \label{algo:edix}
    \begin{itemize}
        \item \textbf{Input:} The curve $\tilde{\mathcal{V}}$ and a prime $\mathfrak{P}$.

        \item \textbf{Output:} The mod-$\mathfrak{P}$ Edixhoven subspace $\mathbb{E}_{\mathfrak{P}}\subset \mathrm{Jac}(\tilde{\mathcal{V}}_{\mathfrak{P}})[\ell]$.
    \end{itemize}
    \begin{algorithmic}[1]

    \STATE \label{algo:edix1} Compute the zeta function $Z(\tilde{\mathcal{V}}_{\mathfrak{P}}/ \mathbb{F}_{\mathfrak{P}}, T)$ by counting points on $\tilde{\mathcal{V}}$ over extensions of $\mathbb{F}_{\mathfrak{P}}$, using a $\mathfrak{P}$ - adic algorithm such as that of Harvey \cite{harvey} (the curve case, specifically, is treated in \cite{kyng}) or Lauder-Wan \cite{LW}.

    \STATE \label{algo:edix2} Compute a basis of each space $$\mathcal{S}_i:=\mathrm{Jac}(\tilde{\mathcal{V}}_{\mathfrak{P}})[\ell](\mathbb{F}_{\mathfrak{P}^{i}})$$ as sums of $\tilde{\mathcal{V}}_{\mathfrak{P}}$ -- points using \cite[Theorem 1]{couv}, with the knowledge of the zeta function, as computed in Step~\ref{algo:edix1}.
    
    \STATE \label{algo:edix3} Compute the $G$ -- action on each subspace $\mathcal{S}_i$. In particular, for each generator $\rho_{\ell}(\sigma_j)$, compute its action as a matrix on a basis of $\mathcal{S}_i$ for each $i$ upto a bound $J=\mathrm{poly}(\ell)$, using~\ref{subsec:geomgal} and \cite[Theorem 1]{couv}. If $G$ does not act on $\mathcal{S}_i$, (i.e., some elements are moved outside it), increment $i$.

    \STATE \label{algo:edix4}  Compute the space of vanishing cycles $M=\mathcal{E}_{\overline{\eta}}\subset \mathcal{F}_{\overline{\eta}}$ with $G$ -- action using  (\ref{subsec:computevan}), and the reduction $M_{\mathfrak{P}}$. Next, compute each element $\phi\in \mathrm{Hom}_{G}(M^{\vee}_{\mathfrak{P}}, \mathcal{S}_i)$ as a matrix, and a basis of the sum of the images. Choosing bases for $M_{\mathfrak{P}}^{\vee}$ and $\mathcal{S}_{i}$, a basis for the space of $G$ -- equivariant homs can be computed by setting up a linear system. In other words, the hom space is given by the maps $\phi$ satisfying the  system $\phi(g\cdot m)=g\cdot \phi(m)$, where $g$ runs over $G$ and $m$ runs over a basis for $M_{\mathfrak{P}}^{\vee}$. Write
    $$
    \mathbb{E}^{(i)}_{\mathfrak{P}}=\sum_{\phi\in \mathrm{Hom}_{G}(M^{\vee}_{\mathfrak{P}}, \mathcal{S}_i)}\mathrm{im}(\phi).
    $$

    \STATE \label{algo:edix5} Compute the invariant space $(\mathbb{E}^{(i)}_{\mathfrak{P}}\otimes_{\mathbb{F}_{\ell}}M_{\mathfrak{P}})^{G}$ and its dimension. If it equals $\beta_2-2$, return $\mathbb{E}^{(i)}_{\mathfrak{P}}$.
        
    \end{algorithmic}
\end{algorithm}

\begin{remark}
    We abuse notation by using $G$ to also refer to the monodromy of the mod-$\mathfrak{P}$ Lefschetz pencil. Provided $\mathfrak{P}$ is large enough compared to the data of the surface, there is an equality between the number of singular fibres in char zero and in positive char. Further, let $u\in \mathcal{U}$ and $\mathsf{u}\in \mathcal{U}_{\mathfrak{P}}$ such that $u\equiv \mathsf{u} \mod \mathfrak{P}$. Let $\overline{\xi}=\mathrm{spec}(\overline{\mathbb{F}_{\mathfrak{P}}(t)})$ be the geometric generic point. Then, we can consistently transport the $G$-action on $\mathcal{F}_{\overline{\eta}}$ to $\mathcal{F}_{\overline{\xi}}$ via the diagram

    \[
\begin{tikzcd}
\mathcal{F}_{\overline{\eta}} \arrow{r}{\phi_{u}^{-1}} \arrow{d}[swap]{} & \mathcal{F}_u \arrow{d}{\varrho_{u}} \\%
\mathcal{F}_{\overline{\xi}} \arrow{r}{\varphi_{\mathsf{u}}^{-1}}& \mathcal{F}_{\mathsf{u}}
\end{tikzcd}
\]
where $\phi_u$ is a choice of cospecialisation at $u$, $\varphi_{\mathsf{u}}$ is the corresponding positive characteristic choice (obtained via coefficient-wise reduction of Laurent series), and $\varrho_u$ is the char-zero to positive-char comparison isomorphism coming from reduction mod $\mathfrak{P}$. Thus, we have an unambiguous $G$ - action on $M_{\mathfrak{P}}=\mathcal{F}_{\overline{\xi}}$. 
\end{remark}

\begin{lemma}
\label{edixlem}
The quantity $J$ in Step~\ref{algo:edix3} of Algorithm~\ref{algo:edix} can be assumed to be bounded by a polynomial in $\ell$.
\end{lemma}

\begin{proof}
    We first show that the Edixhoven subspace $\mathbb{E}_{\mathfrak{P}}\subset \mathrm{Jac}(\tilde{\mathcal{V}}_{\mathfrak{P}})[\ell]$ is defined over a field extension of $\mathbb{F}_{\mathfrak{P}}$ of degree at most bounded by a polynomial in $\ell$. We notice that via its action on the positive characteristic surface $\mathcal{X}_{\mathfrak{P}}$, we have a Galois representation  $$\mathrm{Gal}(\overline{\mathbb{F}}_{\mathfrak{P}}/\mathbb{F}_{\mathfrak{P}}) \rightarrow \mathrm{GL}\left(\mathrm{H}^{2}(\mathcal{X}_{\mathfrak{P}}, \mu_\ell)\right).$$
    The general linear group is of rank $\beta_2$ (the second Betti number of $\mathcal{X}$, which is independent of $\ell$ for most, and indeed any large enough $\ell$, compared to the data of $\mathcal{X}$) over the field $\mathbb{F}_{\ell}$, hence has size bounded by a polynomial in $\ell$. Further, this restricts to an action on $\mathrm{H}^{1}(\mathbb{P}^{1}, \mathcal{F})$. Therefore, by Theorem~\ref{thm:galiso}, it is sufficient to show that $\mathbb{E}_{\mathfrak{P}}$ has dimension independent of $\ell$. Using the tensor-hom duality, we see that 
    $$
    \left(\mathbb{E}_{\mathfrak{P}}\otimes M_{\mathfrak{P}}\right)^{G}\simeq \left(\mathrm{H}^{1}(\tilde{\mathcal{V}}_{\mathfrak{P}}, \mu_{\ell})\otimes M_{\mathfrak{P}}\right)^{G}\simeq \mathrm{Hom}_{G}\left(M_{\mathfrak{P}}^{\vee}, \mathrm{H}^{1}(\tilde{\mathcal{V}}, \mu_{\ell})\right)\simeq \mathrm{Hom}_{G}\left(M_{\mathfrak{P}}^{\vee}, \mathbb{E}_{\mathfrak{P}}\right)
    $$
as $\mathbb{E}_{\mathfrak{P}}$ is the sum of the images of each $\phi\in \mathrm{Hom}_{G}(M^{\vee}_{\mathfrak{P}}, \mathrm{H}^{1}(\tilde{\mathcal{V}}_{\mathfrak{P}}, \mu_\ell))$. The hom space has dimension bounded by $\beta_2$, and the dimension of $M$ is independent of $\ell$, so this shows it.  \\ \\
However, each subspace $\mathcal{S}_i$ may not be mapped to itself under $G$. This is easily tested on elements, after applying $G$ -- action and using the $\mathbb{F}_{\mathfrak{P}^{i}}$ -- frobenius. But, the group $G$ acts on $\tilde{\mathcal{V}}_{\mathfrak{P}}$ via automorphisms defined over an extension $\mathbb{F}_{\mathfrak{P}'}/\mathbb{F}_{\mathfrak{P}}$ of degree at most $\mathrm{poly}(\ell)$, hence in particular, $\mathrm{Jac}(\tilde{\mathcal{V}}_{\mathfrak{P}})[\ell](\mathbb{F}_{\mathfrak{P}'})$ carries a $G$ -- action. In other words, the subspace we are looking for, $\mathbb{E}_{\mathfrak{P}}$, can be found in an $\mathcal{S}_i$ that $G$ does act on, for some $i$ bounded by a polynomial in $\ell$.  The $G$ -- action can then be computed via a basis as in \cite[Theorem 1]{couv}. 
\end{proof}

\begin{remark}
    See also \cite[Lemma 5.6]{leve} for an alternate proof of the fact that the relevant subspace can be found over a poly-bounded extension.
\end{remark}

\begin{theorem}
    \label{thm:edixalgo}
    Algorithm~\ref{algo:edix} outputs the subspace $\mathbb{E}_{\mathfrak{P}}\subset \mathrm{H}^{1}(\tilde{\mathcal{V}}_{\mathfrak{P}}, \mu_\ell)$.
\end{theorem}
\begin{proof}
    We note that $\mathbb{E}_{\mathfrak{P}}$ is the sum of all subspaces of $\mathrm{Jac}(\tilde{\mathcal{V}})[\ell]$ isomorphic to $M_{\mathfrak{P}}$ as $G$-modules. Further, by Lemma~\ref{edixlem}, it can be found within a poly-bounded extension. The algorithm only stops when the invariant subspace has the correct dimension, indicating that we have found the mod-$\mathfrak{P}$ Edixhoven subspace.
\end{proof}
We now indicate how to Hensel- lift torsion points $\mathfrak{P}$ -- adically, following work of Mascot \cite{mascothensel}. We recall the following.

\begin{theorem}[Mascot]
\label{thm:mascot}
 Let $C$  be a model for a nice algebraic curve of genus $g'$ over a number field $L$  given via equations, and let $\rho$ be a mod-$\ell$ $\mathrm{Gal}(\overline{L}/L)$ representation contained in a subspace $S\subset \mathrm{Jac}(C)[\ell]$ of dimension $s$. Let $\mathfrak{P}\subset \mathcal{O}_L$ be a prime of good reduction for $C$ distinct from $\ell$, and assume we are given $P_1(C_{\mathfrak{P}}/\mathbb{F}_{\mathfrak{P}}, T)$ \footnote{i.e., the numerator of the zeta function of $C_{\mathfrak{P}}$}. Further, assume we can isolate the subspace $S_{\mathfrak{P}}\subset \mathrm{Jac}(C_{\mathfrak{P}})[\ell]$. Then, given an accuracy parameter $e$, there exists an algorithm to $\mathfrak{P}$-adically lift the torsion subspace $S_{\mathfrak{P}}$ up to accuracy $\mathfrak{P}^e$, running in time
 $$
 \widetilde{O}(\mathrm{poly}(g'\cdot \log(\#\mathbb{F}_{\mathfrak{P}})\cdot e\cdot \ell^{s})).
 $$
 Further, if the accuracy parameter is sufficient to lift the  subspace to $S$, then the associated $\mathrm{Gal}(\overline{L}/L)$ representation is computed with the same complexity.
\end{theorem}
\begin{proof}
    See \cite[\S 4, 5, 6]{mascothensel}.
\end{proof}
We now give a brief, informal sketch of Mascot's algorithm for completeness, based on the outline \cite[\S 1.2]{mascothensel}. For simplicity, assume the base number field is $\mathbb{Q}$, and we have a rational prime $\mathbf{p}$. 
\begin{itemize}
    \item Compute a basis of $S_{\mathbf{p}}\subset \mathrm{Jac}(C_{\mathbf{p}})[\ell](\mathbb{F}_{\mathbf{q}})$, where $\mathbb{F}_{\mathbf{q}}/\mathbb{F}_{\mathbf{p}}$ is an extension over which the subspace $S_{\mathbf{p}}$ becomes rational.

    \item Given the accuracy parameter $e$, Hensel-lift the basis points to approximation $O(\mathbf{p}^{e})$ in $\mathrm{Jac}(C)(\mathbb{Q}_{\mathbf{q}})$, i.e., points of $\mathrm{Jac}(C)(\mathbb{Z}_{\mathbf{q}}/\mathbf{p}^{e})$.

    \item Compute all the possible $\mathbb{F}_{\ell}$ -- linear combinations of this basis. This is a model of $S$ over $\mathbb{Z}_{\mathbf{q}}/\mathbf{p}^{e}$, consisting of $\ell^{s}$ points.

    \item Write a rational map $\bm{\alpha}: \mathrm{Jac}(C)\dashrightarrow \mathbb{A}^{1}$ defined over the field $\mathbb{Q}$, and evaluate at the $\ell^{s}$ points constructed in the above step. Make sure the values are distinct, else use another rational map.

    \item Form the monic polynomial whose roots are these values and output it.
\end{itemize}
\subsection{Height of divisors in the Edixhoven subspace}
\label{subsec:htedix}
In this subsection, we bound the height of the divisors we are interested in, coming from the Edixhoven subspace. The main estimate is the following.
\begin{theorem}
    \label{thm:edixdivhtbd}
    For each  $x\in \mathbb{E}\subset \mathrm{Jac}(\tilde{\mathcal{V}})[\ell]$, we have
    \begin{equation}
        h(\mathbf{D}_x)\leq \mathrm{poly}(\ell),
    \end{equation}
where $\mathbf{D}_x$ is a representation of the degree zero divisor in $\mathrm{Jac}(\tilde{\mathcal{V}})[\ell]$ corresponding to $x$, as a sum of points in $\tilde{\mathcal{V}}$.
\end{theorem}

\begin{proof}
    We have to show that for $x\in \mathbb{E}\subset \mathrm{Jac}(\tilde{\mathcal{V}})[\ell]$, each point in the support of the divisor representing it, in the framework of Khuri-Makdisi's algorithms \cite{khuri, khuri1} (as used by Mascot), has (logarithmic) Weil height bounded by a polynomial in $\ell$. The strategy is to make use of \cite[Theorem 9.1.3]{ceramanujan} applied to the curve $\tilde{\mathfrak{j}}: \tilde{\mathcal{V}}\rightarrow \mathbb{P}^{1}$. We first note that Theorems 9.1.3, 9.2.1, and 9.2.5 of \cite{ceramanujan} are directly applicable to our setting, as they are concerned with a general algebraic curve or Riemann surface defined over a number field. We address each term in the inequality of \cite[Theorem 9.1.3]{ceramanujan} separately, showing polynomial bounds. \\ \\
    \hfill
    1. \textbf{Faltings height of the curve} $\tilde{\mathcal{V}}$. \\ As a first step, we invoke Theorem~\ref{javanp}, applied to the curve $\tilde{\mathcal{V}}$, which is the normalisation of $\mathbb{P}^{1}$ in the function field of the cover $\mathfrak{j}: \mathcal{V}\rightarrow \mathcal{U}$. Noting that the ramification locus $\mathcal{Z}=\mathbb{P}^{1}\setminus \mathcal{U}$ has cardinality and height depending only on the surface $\mathcal{X}$ and independent of $\ell$, we see that the theorem directly gives that the Faltings height $\mathfrak{h}_F(\tilde{\mathcal{V}})$ of the Jacobian of $\tilde{\mathcal{V}}$ is bounded above by $$
    \deg(\mathfrak{j})^{a},
    $$  
    where the quantity $a$ is independent of $\ell$. Noting that $\deg(\mathfrak{j})$ is bounded by a polynomial in $\ell$ gives the result.
    \\ \\
\hfill
    2. \textbf{Sup norm bounds for the Arakelov-Green's functions} The sup-norm of the Arakelov-Green's functions $g$ is bounded above as a linear function of Faltings' delta invariant $\delta_F(\cdot)$ and the genus $\mathfrak{g}$, by \cite[Corollary 4.6.2]{wilms}. The quantity $\delta_F(\tilde{\mathcal{V}})$ is in turn bounded in Javanpeykar's result \cite[Theorem 6.0.4]{javan}, by a polynomial in $\ell$. \\ \\
\hfill
    3. \textbf{Bounds for the theta function} For the norm of the theta function $\vert \vert \vartheta \vert \vert$ on $\mathrm{Pic}^{\mathfrak{g}-1}(\tilde{\mathcal{V}})$, we have by \cite[Lemma 2.4.2]{javan}
    $$
    \log \vert \vert \vartheta \vert \vert_{\text{max}}\leq \frac{\mathfrak{g}}{4} \log \max(1, \mathfrak{h}_{F}(\tilde{\mathcal{V}}))+ (4\mathfrak{g}^{3}+5\mathfrak{g}+1)\log 2,
    $$
    which is clearly bounded by a polynomial in $\ell$, as both the genus $\mathfrak{g}$ of $\tilde{\mathcal{V}}$ and its Faltings height $\mathfrak{h}_{F}(\tilde{\mathcal{V}})$ are.
    \\ \\
\hfill
    4. \textbf{An integral bound} Consider the integral $$
    \int_{\tilde{\mathcal{V}}}\log(1+ \vert \tilde{\mathfrak{j}} \vert^{2} )\mu_{\tilde{\mathcal{V}}},$$
    where $\mu_{\tilde{\mathcal{V}}}$ is the Arakelov 1-1 form associated to $\tilde{\mathcal{V}}$, regarded as a Riemann surface. By pushing forward to $\mathbb{P}^{1}$, one may conclude a polynomial upper bound for the integral as the degree, the number of poles and (logarithmic) height of the polynomials defining the function $\tilde{\mathfrak{j}}$ are bounded by a polynomial in $\ell$. Further, the ramification locus $\mathcal{Z}\subset \mathbb{P}^{1}$ is independent of $\ell$ as well.
    \\ \\
\hfill
  5. \textbf{Bounds for intersection numbers} For an $\ell$ - torsion divisor $\mathbf{D}_x$ corresponding to $x\in \mathbb{E}$, one can bound the intersection numbers due to work of de Jong \cite[Proposition 2.6.1]{dejong} (see \S 2.6 of loc. cit., more generally, and also \cite[Theorem 9.2.5]{ceramanujan}), combined with the bound for the Arakelov-Green's functions. An explicit version of the estimate is due to Wilms \cite[Propositions 1, 2]{wilms} \footnote{it is not necessary to use Weierstra\ss~points for the algorithm as in \cite{wilms}, however they can anyway be computed efficiently by \cite{hessweiss}} giving polynomial bounds for the intersection numbers using Weierstra\ss~ points.

    With bounds for the above quantities, it follows that for each point $P_x$ in the support of $\mathbf{D}_x$, the absolute Weil height $h(\tilde{\mathfrak{j}}(P_x))$ is bounded by a polynomial in $\ell$, by a similar argument as in \cite[Proposition 11.7.1]{ceramanujan}. This implies the same for $h(P_x)$ as the map $\tilde{\mathfrak{j}}$ itself has height and degree bounded by a polynomial in $\ell$.
\end{proof}
\begin{remark}
    We note that in the proofs of each of the above components, we require $\tilde{\mathcal{V}}$ to be semistable over $K$. This is possible after an extension, but the degree of the extension can be exponential in the genus $\mathfrak{g}$ and hence $\ell$. This does not affect the bounds as the inequalities (in particular, for the intersection number as well) are normalised by the degree $[K: \mathbb{Q}]$, as in \cite[Theorem 9.1.1]{ceramanujan}, ultimately giving polynomial height bounds.
\end{remark}

\begin{remark}
    As an aside, we mention that the result of Javanpeykar, Theorem~\ref{javanp}, provides a heuristic towards Theorem~\ref{thm:edixdivhtbd} in the following sense. An $\ell$-torsion point in $\mathrm{Jac}(\tilde{\mathcal{V}})[\ell]$ is understood as a divisor $\mathbf{D}$, giving a curve $\mathfrak{j}': \mathcal{W}\rightarrow \tilde{\mathcal{V}}$ corresponding to an \'etale $\mu_{\ell}$-torsor. The composite map $$\tilde{\mathfrak{j}}\circ \mathfrak{j}':\mathcal{W}\rightarrow \mathbb{P}^{1}$$ is ramified exactly at $\mathcal{Z}$, and is of degree bounded by a polynomial in $\ell$. Further, the curve $\mathcal{W}$ also has genus bounded by a polynomial in $\ell$ thanks to the Riemann-Hurwitz formula, hence has Faltings height bounded by a polynomial in $\ell$ by Theorem~\ref{javanp}. This suggests that the (logarithmic) Weil height of the algebraic numbers that appear in a ``minimal'' expression for the divisor $\mathbf{D}$ should also likewise be bounded by a polynomial in $\ell$.
\end{remark}
We conclude with the below table, drawing a rough comparison with the leitmotif of the work \cite{ceramanujan}.

\begin{table}[h!]
  \centering
 \renewcommand{\arraystretch}{1.8}
  \begin{tabular}{|p{8cm}|p{8cm}|p{8cm}|}
    \hline
    \textbf{Couveignes-Edixhoven} & \textbf{This work} \\
    \hline
    Modular curve $\mathrm{X}_{1}(5\ell)$ & The curve $\tilde{\mathcal{V}}$ \\
    \hline
    The Ramanujan subspace $\mathrm{V}\subset \mathrm{J}_{1}(5\ell)[\ell] $ & The Edixhoven subspace $\mathbb{E}\subset \mathrm{Jac}(\tilde{\mathcal{V}})[\ell]$ \\
    \hline
   Hecke action to compute $\mathrm{V}$ & Monodromy action to compute $\mathbb{E}$. \\
    \hline
  \end{tabular}
  \caption{Comparison to Couveignes-Edixhoven}
  \label{tab:couvedix}
\end{table}

\section{Main theorem}
\label{subsec:main}
In this section, we state and prove our main result.

\begin{theorem}
\label{thm:main}
    Let $\mathcal{X}$ be a fixed, nice surface of degree $D$ defined over a number field $K$. Then, there exists a randomised algorithm that  
    \begin{itemize}

\item[(i)] on input a prime number $\ell$, outputs the \'etale cohomology groups $\mathrm{H}^{i}(\mathcal{X}, \mu_\ell)$ for $0\leq i\leq 4$ along with the $\mathrm{Gal}(\overline{K}/K)$ action in time
$$\mathrm{poly}(\ell),$$
    
   \item [(ii)] on input a prime $\mathfrak{p}\subset \mathcal{O}_K$ of good reduction with $\mathcal{O}_K/\mathfrak{p}=\mathbb{F}_{q}$, outputs the zeta function of the reduction $Z(X/\mathbb{F}_{q}, T)$, and the point-count $\#X(\mathbb{F}_{q})$ in time
    $$
    \mathrm{poly}(\log q).
    $$

\end{itemize}
\end{theorem}
\begin{proof}

  The computation of the cohomology groups $\mathrm{H}^{i}(\mathcal{X}, \mu_\ell)$ for $i=1,2$ is in Algorithm~\ref{algo:main}. The computation for $i=3$ follows from that of $i=1$ using Poincar\'e duality, while the cases $i=0, 4$ are via suitable twists of the cyclotomic character. The complexity is proved in Lemma~\ref{lem:maincomp}. We remark further, that the output is not dependent on the choice of (co)specialisations in Steps~\ref{algo:main1} and~\ref{algo:main2} of Algorithm~\ref{algo:main}, as ultimately we are interested in monodromy invariants, and any other choices only differ by conjugacy, i.e., the invariant subspaces are always isomorphic as $\mathrm{Gal}(\overline{K}/K)$ modules.
   
   Part (ii) follows in a manner similar to that mentioned in \cite[Remark 1.2]{mascothensel}. One uses an efficient algorithm to compute the image of the Frobenius element at large primes, upto conjugacy, such as \cite{dok}, combined with Section~\ref{app:rec} to recover the zeta function and point count. 
   
\end{proof}
\begin{algorithm}
    \caption{\texttt{Computing the cohomology groups $\mathrm{H}^{i}(\mathcal{X}, \mu_\ell)$}}
    \label{algo:main}
    \begin{itemize}
    \item \textbf{Input:} A smooth projective surface $\mathcal{X}\subset \mathbb{P}^{N}$ of degree $D$ over a number field $K$ presented as a system of homogeneous polynomials of degree $\leq d$ and a prime number $\ell$.

    \item \textbf{Pre-processing:} Fibre $\mathcal{X}$ as a Lefschetz pencil $\pi:\mathcal{X}\rightarrow \mathbb{P}^{1}$. Let $\mathcal{Z}\subset \mathbb{P}^{1}$ parametrise the singular fibres and $\mathcal{U}=\mathbb{P}^{1}\setminus Z$ the smooth ones. Embed the Jacobian of the generic fibre $\mathcal{X}_{\overline{\eta}}$ into $\mathbb{P}^{M}$ obtaining the $\ell$ -- torsion $\mathrm{Pic}^{0}(\mathcal{X}_{\overline{\eta}})[\ell]$ as the $\overline{K(t)}$ -- roots of the ideal $\prescript{(\ell)}{}{}\mathcal{I}_{\overline{\eta}}$ using Algorithm~\ref{algo:genelldiv}.
  
    \item \textbf{Output:} The cohomology groups $\mathrm{H}^{i}(\mathcal{X}, \mu_\ell)$ for $1\leq i\leq  2$ presented as $\mathbb{F}_{\ell} $ -- vector spaces with bases and $\mathrm{Gal}(\overline{K}/K) $ -- action.
    
\end{itemize}

\begin{algorithmic}[1]

\STATE \label{algo:main1} Choose a point $u\in \mathcal{U}({K})$ of bounded height and degree, to serve as base point.

\STATE \label{algo:main2} Compute a cospecialisation $\phi_u:\mathcal{F}_u\rightarrow \mathcal{F}_{\overline{\eta}}$, by making a choice of expansion for the primitive element $\bm{\tau}$ around $u$, hence obtaining each $\gamma\in \mathcal{F}_{\overline{\eta}}$ as Laurent series around $u$.

\STATE \label{algo:main3} Compute the image of the monodromy fixed subspace, i.e., those elements $\gamma\in \mathcal{F}_{\overline{\eta}}$ fixed by each $\sigma_j$ for $1\leq j<r$, with the monodromy action as computed in Algorithm~\ref{algo:mono}.

\STATE \label{algo:main4} Compute the Galois action on the monodromy fixed subspace $\mathcal{F}_u^{G}:=\phi_u^{-1}(\mathcal{F}_{\overline{\eta}}^{G})$ element-wise, using the cospecialisation $\phi_u$. This gives $\mathrm{H}^{1}(\mathcal{X}, \mu_\ell)$ with $\mathrm{Gal}(\overline{K}/K)$ action.

\STATE \label{algo:main5}  Compute the subspace $\mathcal{E}_{u}$ as the complement $(\mathcal{F}_{u}^{G})^{\perp}$ under the symplectic Weil pairing on $\mathcal{F}_u$. 

\STATE \label{algo:main6} Choose an auxiliary small prime of good reduction $\mathfrak{P}$, with characteristic at most of size $O(\ell)$, distinct from $\ell$. Now, for the second cohomology, we work with the curve $\tilde{\mathcal{V}}$. Reduce modulo $\mathfrak{P}$ and compute the subspace $\mathbb{E}_{\mathfrak{P}}\subset \mathrm{H}^{1}(\tilde{\mathcal{V}}_{\mathfrak{P}}, \mu_\ell)$ using Algorithm~\ref{algo:edix}.
\STATE \label{algo:main7} Lift the subspace $\mathbb{E}_{\mathfrak{P}}$ to the characteristic zero subspace $\mathbb{E}\subset \mathrm{H}^{1}(\tilde{\mathcal{V}}, \mu_\ell)$ using Theorem~\ref{thm:mascot}.

\STATE \label{algo:main8} Compute the space of invariant tensors $$(\mathbb{E}\otimes \mathcal{E}_u)^{G}$$ with knowledge of the $G$ - action.

\STATE \label{algo:main9} Compute the diagonal $\mathrm{Gal}(\overline{K}/K)$ action as a matrix on the subspace of tensors which has been isolated in the above step, element-wise. This gives the space $\mathrm{H}^{1}(\mathbb{P}^{1}, \mathcal{F})$ with $\mathrm{Gal}(\overline{K}/K)$ - action. To obtain the full $\mathrm{H}^{2}(\mathcal{X}, \mu_\ell)$, we just add the space $<\gamma_E>\oplus <\gamma_F>$, on which Galois acts via the cyclotomic character on each component.

\end{algorithmic}
\end{algorithm}

\section{Complexity analyses}
\label{sec:comp}
In this section, we prove the upper bounds for the complexities stated of the subroutines used in the earlier sections. We do not deduce the exact complexities beyond showing that they are bounded by polynomial functions of $\ell$ and $\log q$. 
We also keep track of the heights of the algebraic numbers involved in the computations. 

\subsection{Algorithms of Sections~\ref{sec:prelim} and~\ref{sec:sub}}
\label{subsec:prelimcomp}
Noting that the complexity of Algorithm~\ref{algo:blowup} is independent of $\ell$, we begin with the following.

\begin{lemma}
    Algorithm~\ref{algo:genelldiv} runs in time $\mathrm{poly}(\ell)$.
\end{lemma}
\begin{proof}
Pila \cite[\S 2]{pila} shows that the data representing the multiplication by $\ell$ map is bounded by a polynomial in $\ell$. Further, the coefficients occurring in the ideal $\prescript{(\ell)}{}{}\mathcal{I}_{\overline{\eta}}$ have height bounded by a polynomial in $\ell$ due to Theorem~\ref{thm:tors_height} and the fact that the Faltings height of the (normalisation of the) curve  $\prescript{(\ell)}{}{}\mathfrak{C}$ over $K$ given by $\prescript{(\ell)}{}{}\mathcal{I}_{\overline{\eta}}$ is bounded by a polynomial in $\ell$ \cite[Theorem 6.0.6]{javan}.

\end{proof}


\begin{lemma}
    Algorithm~\ref{algo:cospec} runs in time $\mathrm{poly}(\ell)$.
\end{lemma}

\begin{proof}
\hfill
\begin{itemize}
    \item Step~\ref{algo:cospec1}: The complexity of Algorithm~\ref{algo:genelldiv} has been shown to be polynomial in $\ell$.

    \item Step~\ref{algo:cospec2}: Zero-dimensional system solving can be done using a primitive element in time polynomial in the degree of the system by \cite{rouillier}.

    \item Step~\ref{algo:cospec3}: Computing the first $m$ coefficients of a branch can be done in $\mathrm{poly}(m)$ time by Theorem~\ref{thm:ntpu}. It suffices to compute the first $\mathrm{poly}(\ell)$ coefficients to uniquely specify a branch by Lemma~\ref{lem:puibound}.

    \item Step~\ref{algo:cospec4}: Once a choice of Puiseux series for $\bm{\tau}$ is made, simple arithmetic (addition, squaring) can be performed using it in polynomial time.

\end{itemize}

\end{proof}

\begin{lemma}
    Algorithm~\ref{algo:recenter} runs in time $\mathrm{poly}(\ell)$.
\end{lemma}

\begin{proof}
\hfill

\begin{itemize}
    \item Step 1: Specialisation of the ideal $\prescript{(\ell)}{}{}\mathcal{I}_{\overline{\eta}}$  to $u$ mearly involves making the substitution $t=u$.

    \item Step 2: The specialised ideal $\prescript{(\ell)}{}{}\mathcal{I}_{u}$ is now zero-dimensional over $\overline{K}$ and its roots can be found by a system solver \cite{rouillier}. The Weil height of the $\ell$ -- torsion points is bounded by a polynomial in $\ell$ by Theorem~\ref{thm:tors_height}.

    \item Step 3: Convergence to an algebraic number with $\mathrm{poly}(\ell)$ precision is guaranteed by Theorem~\ref{thm:approx}.
\end{itemize}

\end{proof}

\begin{lemma}
    
    Algorithm~\ref{algo:van} runs in time $\mathrm{poly}(\ell)$.
\end{lemma}
\begin{proof} \hfill

\begin{itemize}
   \item Step~\ref{algo:van1}: Follows from the complexity of Algorithm~\ref{algo:cospec}.

   \item Step~\ref{algo:van2}: An element $\gamma\in \mathcal{F}_{\overline{\eta}}\setminus \phi_{z}(\mathcal{F}_{z})$ can be chosen by ensuring that at least one of the tuple of Puiseux expansions associated to $\gamma$ is ramified at $z$, i.e., is in fact belongs to $\overline{K}\langle \langle t-z \rangle \rangle \setminus \overline{K}((t-z))$.

   \item Step~\ref{algo:van3}: As each Puiseux expansion is specified only upto the first $\mathrm{poly}(\ell)$ coefficients by Lemma~\ref{lem:puibound}, one has to simply multiply each (non-constant) coefficient by a power of $\zeta_{\ell}$.

   \item Steps~\ref{algo:van4} \& ~\ref{algo:van5}: The complexity follows from that of Algorithm~\ref{algo:recenter}.

   \item Step~\ref{algo:van6}: The addition of the group law can be performed efficiently by Theorem~\ref{thm:eqnjac}.

   \item Step~\ref{algo:van7}: The complexity of computing the abstract Abel map and its inverse (Algorithm~\ref{algo:abel}) is given by Theorem~\ref{thm:eqnjac}.

   \item Step~\ref{algo:van8}: Pairings can be computed in polynomial time using a divisorial description by Algorithm~\ref{algo:pairing}.

   \item Step~\ref{algo:van9}: Square root over $\mathbb{Z}/\ell \mathbb{Z}$ can be found in randomised polynomial time.

\item Step~\ref{algo:van10}: The rational functions in $\bm{\tau}$ corresponding to Puiseux expansions around $z$, can be found in polynomial time via linear algebra combined with $\mathrm{poly}(\ell)$ truncations.
    \end{itemize}
\end{proof}
\begin{lemma}
    Algorithm~\ref{algo:mono} runs in time $\mathrm{poly}(\ell)$.
\end{lemma}
\begin{proof}



\hfill

\begin{itemize}
    \item Step~\ref{algo:mono1}: Follows from the complexity of Algorithm~\ref{algo:van}.

        \item Steps ~\ref{algo:mono2} and ~\ref{algo:mono3} : This boils down to the problem of expressing elements in the splitting field of the $\ell$ -- torsion of the Jacobian of the generic fibre, as rational functions in a primitive element for the field extension. This can be solved on the level of Puiseux series as well, with $\mathrm{poly}(\ell)$ truncation, by Lemma~\ref{lem:puibound}.

\end{itemize}

\end{proof}

\subsection{Algorithms of Sections~\ref{sec:algos} and~\ref{subsec:main}}
\label{subsec:algoscomp}

\begin{lemma}
Algorithm~\ref{algo:edix} runs in time $\mathrm{poly}(\ell \cdot \mathrm{char}(\mathbb{F}_{\mathfrak{P}})\cdot \log(\# \mathbb{F}_{\mathfrak{P}}))$.
\end{lemma}
\begin{proof}
\hfill
    \begin{itemize}
        \item Step~\ref{algo:edix1}: One can use a $\mathfrak{P}$ -- adic algorithm such as that of Harvey \cite{harvey} \footnote{the curve case, specifically, is dealt with in \cite{kyng}} or Lauder-Wan \cite{LW} to count points on $\tilde{\mathcal{V}}_{\mathfrak{P}}$ to output its zeta function with the stated complexity. It is sufficient to count points upto an extension of degree bounded by the genus $\mathfrak{g}$, which in this case is bounded by a polynomial in $\ell$.
        
         \item Step~\ref{algo:edix2}: A basis for the space $\mathcal{S}_i$ can be computed in polynomial time using random sampling on the curve, following \cite[Theorem 1]{couv}, with knowledge of the zeta function.
        \item Step~\ref{algo:edix3}: The $G$ - action is computed on the points of the curve $\tilde{\mathcal{V}}_{\mathfrak{P}}$ following~\ref{subsec:geomgal} in polynomial time. The number $J$ is bounded by a polynomial in $\ell$ by Lemma~\ref{edixlem}. Further, given a basis of each subspace $\mathcal{S}_i$, the $G$ -- action can be computed in polynomial time on the basis, following the last part of \cite[Theorem 1]{couv}.
        
        \item Step~\ref{algo:edix4}: Computing $M$ is possible in polynomial time with $G$ -- action by the algorithms of (\ref{subsec:computevan}). The reduction can also be computed in poly-time. Next, the $G$ - action on the dual $M^{\vee}_{\mathfrak{P}}$ can be computed using the action on $M_{\mathfrak{P}}$ for the $G$ -- action via the duality given by the symplectic Weil pairing. Equivalently, given the $G$ -- action on $M_{\mathfrak{P}}$, there is a natural $G$ -- action on $M_{\mathfrak{P}}^{\vee}$ via $(g\cdot \lambda) (m)=\lambda(g^{-1}\cdot m)$ for $\lambda \in M_{\mathfrak{P}}^{\vee}$ and $m\in M_{\mathfrak{P}}$.
        
        Next, the dimension of the space $\mathrm{Hom}_{G}(M_{\mathfrak{P}}^{\vee}, \mathcal{S}_i)$ is bounded independently of $\ell$, and each $G$ -- equivariant homomorphism can be computed as a matrix via linear algebra.
        
        In other words, there are only $\mathrm{poly}(\ell)$ homs, and a basis for the sum of their images can be found using \cite[Theorem 1]{couv}.

        \item Step~\ref{algo:edix5}: One can list all the invariant tensors with knowledge of the $G$ -- action. Further, zero-testing is efficient and can be done in polynomial time, so it simply remains to count the number of invariant tensors in each space, which is always bounded by a polynomial in $\ell$. Finally the Betti number $\beta_2$ can be computed as $\#\mathcal{Z}+2\beta_1+2-4g$, where $g$ is the genus of the generic fibre of the pencil.

    \end{itemize}
\end{proof}

\begin{lemma}
\label{lem:maincomp}
Algorithm~\ref{algo:main} runs in time $\mathrm{poly}(\ell)$.
\end{lemma}
\begin{proof}
\hfill
\begin{itemize}
\item Step~\ref{algo:main1}: This can be done in polynomial time.
\item Step~\ref{algo:main2}: The complexity is the same as that of computing a Puiseux series expansion, except we are now working around a smooth point. The $\mathrm{poly}(\ell)$ truncation bounds remain, and the total complexity is the same as that of Algorithm~\ref{algo:cospec}.
\item Step~\ref{algo:main3}: Follows from the complexity of Algorithm~\ref{algo:mono}.
\item Step~\ref{algo:main4}: The arithmetic $\mathrm{Gal}(\overline{K}/K)$ action on $\mathcal{F}_{u}^{G}$ factors via a finite extension $K'/K$ that the subspace is rational over. This extension has degree bounded by a polynomial in $\ell$, as its Galois group is a subgroup of $\mathrm{GL}(\mathcal{F}_u^{G})$, whose rank is independent of $\ell$.
\item Step~\ref{algo:main5}: Computing the Weil pairing on $\mathcal{F}_u$ has a polynomial time algorithm.

\item Step~\ref{algo:main6}: Follows from the complexity of Algorithm~\ref{algo:edix}. A key point, as mentioned earlier, is that $\dim \mathbb{E}_{\mathfrak{P}}$ is independent of $\ell$.
\item Step~\ref{algo:main7}: Follows from \cite{mascothensel}, i.e., the complextiy of Theorem~\ref{thm:mascot}. The preceision $e$ required depends on the complexity of the algebraic numbers occurring in an explicit description of the Edixhoven subspace. We know by Theorem~\ref{thm:edixdivhtbd} that the heights are bounded by a polynomial in $\ell$. Further, the points occurring in the support of the divisors concerned, each also have degree bounded by a polynomial in $\ell$, as the Edixhoven subspace becomes rational over such an extension.

\item Step~\ref{algo:main8}: The $G$- action on the space of tensors $\mathbb{E}\otimes \mathcal{E}_u$ can be computed element by element, as its dimension is now independent of $\ell$. 
\item Step~\ref{algo:main9}: Again the $\mathrm{Gal}(\overline{K}/K)$ action factors through a finite extension $K''/K$, with degree bounded by a polynomial in $\ell$. Its action on   $\mathbb{E}$ is obtained via points on $\tilde{\mathcal{V}}$, and the action on $\mathcal{E}_u$ can be computed akin to Step~\ref{algo:main4}.

\end{itemize}
\end{proof}

\section{Conclusion}
In this article, we have provided an algorithm to compute the number of points on a fixed, nice surface in polynomial time, having made its \'etale cohomology groups explicit. An area for improvement would be the dependence of the total complexity on the degree of the surface which is, at the moment, exponential. In another direction, one could ask if in the realm of \textit{quantum} algorithms, the dependence on the degree could be made polynomial. The immediate next question, with regard to point counting, is that of algorithms for varieties of a higher dimension, to begin with, threefolds. Specifically, one would need a method to compute vanishing cycles, the Poincar\'e duality pairing in $\mathrm{H}^{2}$ and the corresponding trivialising cover.



\section*{Acknowledgements}
We thank Jean-Pierre Serre for comments on an earlier draft, particularly with regard to his question from \cite{serre}. We thank Joe Silverman and Robin de Jong for pointing us to the literature on heights. We thank Robin de Jong for detailed comments on an earlier draft, that has significantly improved our exposition. We thank Felipe Voloch for discussions leading to the discovery of a gap in an earlier version of this manuscript. We are grateful to T.N. Venkataramana, Arvind Nair and Kiran Kedlaya for discussions, and the institutions Ashoka University, ICTS, TIFR and MIT for their hospitality. We thank the organisers and attendees of the special sessions on `Arithmetic geometry \& Number theory' and `Computational number theory' at the Joint Meeting of the NZMS, AustMS and AMS, held in Auckland, for their feedback on a preliminary talk based on this work. This research work was partially supported by the Research-I Foundation of the Department of Computer Science \& Engineering of IIT-Kanpur. N.S. thanks the funding support from DST-SERB (CRG/2020/45 + JCB/2022/57) and N. Rama Rao Chair.  M.V. is supported by a C3iHub research fellowship.

\bibliographystyle{alpha}
\bibliography{surfaces.bib}

\newcommand{\etalchar}[1]{$^{#1}$}
\begin{thebibliography}{LPPV24}

\bibitem[And97]{andersonabel}
Greg~W Anderson.
\newblock {An explicit algebraic representation of the Abel map.}
\newblock {\em IMRN: International Mathematics Research Notices}, 1997(11), 1997.

\bibitem[And02]{anderson}
Greg~W Anderson.
\newblock {Abeliants and their application to an elementary construction of Jacobians}.
\newblock {\em Advances in Mathematics}, 172(2):169--205, 2002.

\bibitem[And04]{andersontheta}
Greg~W Anderson.
\newblock {Edited 4-$\Theta$ embeddings of Jacobians}.
\newblock {\em Michigan Mathematical Journal}, 52(2):309--339, 2004.

\bibitem[Bug04]{bugeaudbook}
Yann Bugeaud.
\newblock {\em Approximation by algebraic numbers}.
\newblock Cambridge University Press, 2004.

\bibitem[CE11]{ceramanujan}
Jean-Marc Couveignes and Bas Edixhoven.
\newblock {\em Computational aspects of modular forms and Galois representations}.
\newblock Princeton University Press, 2011.

\bibitem[CF{\etalchar{+}}12]{handbook}
Henri Cohen, Gerhard Frey, et~al.
\newblock {\em Handbook of Elliptic and Hyperelliptic Curve Cryptography, Second Edition}.
\newblock Chapman \& Hall/CRC, 2nd edition, 2012.

\bibitem[Cho54]{chow}
Wei-Liang Chow.
\newblock {The Jacobian variety of an algebraic curve}.
\newblock {\em American Journal of Mathematics}, 76(2):453--476, 1954.

\bibitem[Cou06]{hhs}
Jean-Marc Couveignes.
\newblock {Hard Homogeneous Spaces}.
\newblock Cryptology {ePrint} Archive, Paper 2006/291, 2006.

\bibitem[Cou09]{couv}
Jean-Marc Couveignes.
\newblock {Linearizing torsion classes in the Picard group of algebraic curves over finite fields}.
\newblock {\em Journal of Algebra}, 321(8):2085--2118, 2009.

\bibitem[CTS21]{cts}
Jean-Louis Colliot-Th{\'e}l{\`e}ne and Alexei~N Skorobogatov.
\newblock {\em {The Brauer-Grothendieck group}}, volume~71.
\newblock Springer, 2021.

\bibitem[DD13]{dok}
Tim Dokchitser and Vladimir Dokchitser.
\newblock {Identifying Frobenius elements in Galois groups}.
\newblock {\em Algebra \& Number Theory}, 7(6):1325--1352, 2013.

\bibitem[Del74]{Weili}
Pierre Deligne.
\newblock La conjecture de {Weil} : {I}.
\newblock {\em Publications Math\'ematiques de l'IH\'ES}, 43:273--307, 1974.

\bibitem[dJ04]{dejong}
Robin~S de~Jong.
\newblock {\em {Explicit Arakelov Geometry}}.
\newblock PhD thesis, Universiteit van Amsterdam, 2004.

\bibitem[EDJS10]{edcovers}
Bas Edixhoven, Robin De~Jong, and Jan Schepers.
\newblock Covers of surfaces with fixed branch locus.
\newblock {\em International Journal of Mathematics}, 21(07):859--874, 2010.

\bibitem[Gab83]{gabber}
Ofer Gabber.
\newblock Sur la torsion dans la cohomologie $\ell$-adique d’une vari{\'e}t{\'e}.
\newblock {\em CR Acad. Sci. Paris S{\'e}r. I Math}, 297(3):179--182, 1983.

\bibitem[Gel84]{langlands}
Stephen Gelbart.
\newblock {An elementary introduction to the Langlands program}.
\newblock {\em Bulletin of the American Mathematical Society}, 10(2):177--219, 1984.

\bibitem[Har15]{harvey}
David Harvey.
\newblock Computing zeta functions of arithmetic schemes.
\newblock {\em Proceedings of the London Mathematical Society}, 111(6):1379--1401, 2015.

\bibitem[Hes02]{hessweiss}
Florian Hess.
\newblock {An algorithm for computing Weierstrass points}.
\newblock In {\em International Algorithmic Number Theory Symposium}, pages 357--371. Springer, 2002.

\bibitem[HI94]{hirr}
Ming-Deh Huang and Doug Ierardi.
\newblock {Efficient algorithms for the Riemann-Roch problem and for addition in the Jacobian of a curve}.
\newblock {\em Journal of Symbolic Computation}, 18(6):519--539, 1994.

\bibitem[HI98]{huaier}
Ming-Deh Huang and Doug Ierardi.
\newblock Counting points on curves over finite fields.
\newblock {\em Journal of Symbolic Computation}, 25(1):1--21, 1998.

\bibitem[HM17]{hickel}
Michel Hickel and Micka{\"e}l Matusinski.
\newblock {On the algebraicity of Puiseux series}.
\newblock {\em Revista Matem{\'a}tica Complutense}, 30:589--620, 2017.

\bibitem[HS83]{hilliker}
David~Lee Hilliker and EG~Straus.
\newblock {Determination of bounds for the solutions to those binary Diophantine equations that satisfy the hypotheses of Runge’s theorem}.
\newblock {\em Transactions of the American Mathematical Society}, 280(2):637--657, 1983.

\bibitem[Igu56a]{igusa}
Jun-Ichi Igusa.
\newblock {Fibre systems of Jacobian varieties}.
\newblock {\em American Journal of Mathematics}, 78(1):171--199, 1956.

\bibitem[Igu56b]{igusaii}
Jun-Ichi Igusa.
\newblock {Fibre Systems of Jacobian Varieties:(II. Local Monodromy Groups of Fibre Systems)}.
\newblock {\em American Journal of Mathematics}, 78(4):745--760, 1956.

\bibitem[Igu58]{igusaabs}
Jun-Ichi Igusa.
\newblock Abstract vanishing cycle theory.
\newblock {\em Proceedings of the Japan Academy}, 34(9):589--593, 1958.

\bibitem[Jav14]{javan}
Ariyan Javanpeykar.
\newblock {Polynomial bounds for Arakelov invariants of Belyi curves}.
\newblock {\em Algebra \& Number Theory}, 8(1):89--140, 2014.

\bibitem[JP76]{jones}
Wayne Jones and Brian Parshall.
\newblock {On the 1-cohomology of finite groups of Lie type}.
\newblock In {\em Proceedings of the Conference on Finite Groups}, pages 313--328. Elsevier, 1976.

\bibitem[Ked06]{ked}
Kiran~S Kedlaya.
\newblock Quantum computation of zeta functions of curves.
\newblock {\em computational complexity}, 15(1):1--19, 2006.

\bibitem[KM04]{khuri1}
Kamal Khuri-Makdisi.
\newblock Linear algebra algorithms for divisors on an algebraic curve.
\newblock {\em Mathematics of Computation}, 73(245):333--357, 2004.

\bibitem[KM07]{khuri}
Kamal Khuri-Makdisi.
\newblock {Asymptotically fast group operations on Jacobians of general curves}.
\newblock {\em Mathematics of Computation}, 76(260):2213--2239, 2007.

\bibitem[Koz94]{Kozen}
Dexter Kozen.
\newblock {Efficient Resolution of Singularities of Plane Curves}.
\newblock {\em Foundation of Software Technology and Theoretical Computer Science}, 141:1 -- 11, 1994.

\bibitem[KV25]{gcd}
Hyuk~Jun Kweon and Madhavan Venkatesh.
\newblock Bornes de torsion et un th{\'e}or{\`e}me effectif du pgcd.
\newblock {\em arXiv preprint arXiv:2511.00431}, 2025.

\bibitem[Kyn22]{kyng}
Madeleine Kyng.
\newblock {Computing zeta functions of algebraic curves using Harvey’s trace formula}.
\newblock {\em Research in Number Theory}, 8(4):100, 2022.

\bibitem[Lev22]{levrat}
Christophe Levrat.
\newblock Calcul effectif de la cohomologie des faisceaux constructibles sur le site \'etale d'une courbe.
\newblock {\em arXiv:2209.10221}, 2022.

\bibitem[Lev24]{leve}
Christophe Levrat.
\newblock Computing the cohomology of constructible \'etale sheaves on curves.
\newblock {\em Journal de th\'eorie des nombres de Bordeaux}, 36(3):1085--1122, 2024.

\bibitem[LGS20]{lespa}
Aude Le~Gluher and Pierre-Jean Spaenlehauer.
\newblock {A fast randomized geometric algorithm for computing Riemann-Roch spaces}.
\newblock {\em Mathematics of Computation}, 89(325):2399--2433, 2020.

\bibitem[LPPV24]{lpv}
Pierre Lairez, Eric Pichon-Pharabod, and Pierre Vanhove.
\newblock Effective homology and periods of complex projective hypersurfaces.
\newblock {\em Mathematics of Computation}, 93(350):2985--3025, 2024.

\bibitem[LR87]{lr}
Robert~P Langlands and Michael Rapoport.
\newblock {Shimuravariet{\"a}ten und Gerben.}
\newblock {\em Journal f{\"u}r die reine und angewandte Mathematik}, 378:113--220, 1987.

\bibitem[LR10]{lr2}
David Lubicz and Damien Robert.
\newblock Efficient pairing computation with theta functions.
\newblock In {\em International Algorithmic Number Theory Symposium}, pages 251--269. Springer, 2010.

\bibitem[LR15]{lrpairing}
David Lubicz and Damien Robert.
\newblock {A generalisation of Miller's algorithm and applications to pairing computations on abelian varieties}.
\newblock {\em Journal of Symbolic Computation}, 67:68--92, 2015.

\bibitem[LW06]{LW}
Alan~G.B. Lauder and Daqing Wan.
\newblock Counting points on varieties over finite fields of small characteristic.
\newblock {\em Algorithmic Number Theory: Lattices, Number Fields, Curves and Cryptography}, 2006.

\bibitem[Mas20]{mascothensel}
Nicolas Mascot.
\newblock {Hensel-lifting torsion points on Jacobians and Galois representations}.
\newblock {\em Mathematics of Computation}, 89(323):1417--1455, 2020.

\bibitem[Mas23a]{mascot2023explicit}
Nicolas Mascot.
\newblock {Explicit Computation of a Galois Representation Attached to an Eigenform Over SL 3 from the {$\mathrm{H}^{2}_{\text{\'et}}$} of a Surface}.
\newblock {\em Foundations of Computational Mathematics}, 23(2):519--543, 2023.

\bibitem[Mas23b]{mascot}
Nicolas Mascot.
\newblock {Explicit computation of Galois representations occurring in families of curves}.
\newblock {\em arXiv preprint arXiv:2304.04701}, 2023.

\bibitem[Mil80]{milne80}
James~S Milne.
\newblock {\em Etale cohomology (PMS-33)}.
\newblock {Princeton University Press}, 1980.

\bibitem[Mil98]{milnelec}
James~S Milne.
\newblock Lectures on {\'e}tale cohomology.
\newblock Available on-line at \url{http://www.jmilne.org/math/CourseNotes/LEC.pdf}, 1998.

\bibitem[MO15]{mad}
David Madore and Fabrice Orgogozo.
\newblock Calculabilit{\'e} de la cohomologie {\'e}tale modulo $\ell$.
\newblock {\em {Algebra \& Number Theory}}, 9(7):1647--1739, 2015.

\bibitem[Mor00]{moriwaki}
Atsushi Moriwaki.
\newblock Arithmetic height functions over finitely generated fields.
\newblock {\em Inventiones mathematicae}, 140:101--142, 2000.

\bibitem[Pil90]{pila}
Jonathan Pila.
\newblock Frobenius maps of abelian varieties and finding roots of unity in finite fields.
\newblock {\em Mathematics of Computation}, 55(192):745--763, 1990.

\bibitem[PTvL15]{ptv}
Bjorn Poonen, Damiano Testa, and Ronald van Luijk.
\newblock {Computing N{\'e}ron--Severi groups and cycle class groups}.
\newblock {\em Compositio Mathematica}, 151(4):713--734, 2015.

\bibitem[PW21]{pazuki}
Fabien Pazuki and Martin Widmer.
\newblock {Bertini and Northcott}.
\newblock {\em Research in Number Theory}, 7:1--18, 2021.

\bibitem[R{\'e}m10]{remond}
Ga{\"e}l R{\'e}mond.
\newblock Nombre de points rationnels des courbes.
\newblock {\em Proceedings of the London Mathematical Society}, 101(3):759--794, 2010.

\bibitem[Rou99]{rouillier}
Fabrice Rouillier.
\newblock Solving zero-dimensional systems through the rational univariate representation.
\newblock {\em Applicable Algebra in Engineering, Communication and Computing}, 9(5):433--461, 1999.

\bibitem[RSV25]{rsv}
Diptajit Roy, Nitin Saxena, and Madhavan Venkatesh.
\newblock Complexity of counting points on curves and the factor $ p\_1 (t) $ of the zeta function of surfaces.
\newblock {\em arXiv preprint arXiv:2511.02262}, 2025.

\bibitem[Sch85]{schoof}
Ren{\'e} Schoof.
\newblock Elliptic curves over finite fields and the computation of square roots mod $p$.
\newblock {\em Mathematics of computation}, 44(170):483--494, 1985.

\bibitem[Ser12]{serrealg}
Jean-Pierre Serre.
\newblock {\em Algebraic groups and class fields}, volume 117.
\newblock Springer Science \& Business Media, 2012.

\bibitem[Ser16]{serre}
Jean-Pierre Serre.
\newblock {\em {Lectures on $N_X(p)$}}.
\newblock CRC Press, 2016.

\bibitem[Sil83]{silverman}
Joseph~H. Silverman.
\newblock Heights and the specialization map for families of abelian varieties.
\newblock {\em Journal für die reine und angewandte Mathematik}, 342:197--211, 1983.

\bibitem[Wal00]{walsh}
P~Walsh.
\newblock {A polynomial-time complexity bound for the computation of the singular part of a Puiseux expansion of an algebraic function}.
\newblock {\em Mathematics of Computation}, 69(231):1167--1182, 2000.

\bibitem[Wal04]{wall}
C.~T.~C. Wall.
\newblock {\em Singular Points of Plane Curves}.
\newblock London Mathematical Society Student Texts. Cambridge University Press, 2004.

\bibitem[Wil16]{wilms}
Robert Wilms.
\newblock {\em {The delta invariant in Arakelov geometry}}.
\newblock PhD thesis, Universit{\"a}ts-und Landesbibliothek Bonn, 2016.

\bibitem[ZM72]{zarhinmanin}
Ju~G Zarhin and Ju~I Manin.
\newblock Height on families of abelian varieties.
\newblock {\em Mathematics of the USSR-Sbornik}, 18(2):169, 1972.

\end{thebibliography}

\appendix
\section{Recovering zeta}
\label{app:rec}
The objective of this section of the appendix is to show how to recover the zeta function of a smooth, projective surface from the action of Frobenius on its \'etale cohomology groups. As usual, let $X\subset \mathbb{P}^{N}$ be a nice surface of degree $D$ obtained via good reduction from a nice surface $\mathcal{X}$  over a number field $K$, at  a prime $\mathfrak{p}\subset \mathcal{O}_{K}$.
Assume we have computed the action of the Frobenius endomorphism $F_{q}^{\star}$ on the cohomology groups $\mathrm{H}^{i}(X, \mathbb{Z}/\ell \mathbb{Z})$ for $0\leq i \leq 4$. We show how to recover the zeta function $Z(X/\mathbb{F}_{q}, T)$ and the point-count $\#X(\mathbb{F}_{q})$ as follows. Firstly, denote $\tilde{P}_{i}(T):=\det\left(1-TF_{q}^{\star
} \ \vert \ \mathrm{H}^{i}(X, \mathbb{Z}/\ell \mathbb{Z})\right)\in \mathbb{F}_{\ell}[T]$.
Consider the following exact sequence of \'etale sheaves on $X$ following \cite{gabber}
$$
0\longrightarrow \mathbb{Z}_{\ell}\longrightarrow \mathbb{Z}_{\ell}\longrightarrow \mathbb{Z}/\ell \mathbb{Z}\longrightarrow 0.
$$
As a result, we obtain the following from the associated long-exact-sequence on cohomology
\begin{equation}
\label{eqn:torcoho}
0\longrightarrow \mathrm{H}^{i}(X, \mathbb{Z}_{\ell})/(\ell \cdot \mathrm{H}^{i}(X, \mathbb{Z}_{\ell}))\longrightarrow \mathrm{H}^{i}(X, \mathbb{Z}/\ell \mathbb{Z})\longrightarrow \mathrm{H}^{i+1}(X, \mathbb{Z}_{\ell})[\ell]\longrightarrow 0.
\end{equation}
Writing 

$$P_{i}'(T):=\det\left(1- TF_{q}^{\star} \ \vert \ \mathrm{H}^{i}(X, \mathbb{Z}_{\ell})[\ell]\right) \  \text{and} \ \overline{P}_{i}(T):= \det\left(1- TF_{q}^{\star} \ \vert \ \mathrm{H}^{i}(X, \mathbb{Q}_{\ell})\right) \text{mod} \ \ell , $$
we see from (\ref{eqn:torcoho}) that
$$
\tilde{P}_{i}(T)=\overline{P}_{i}(T)
\cdot P_{i}'(T) \cdot P_{i+1}' (T).
$$
In particular if we write $Z(X/ \mathbb{F}_{q}, T)=P(T)/Q(T)$ for $P(T), Q(T)\in \mathbb{Z}[T]$, we see that
$$
\frac{\overline{P}(T)}{\overline{Q}(T)}=\prod_{i=0}^{4}(\tilde{P}_{i}(T))^{(-1)^{i+1}}
$$
where $\overline{P}(T):=P(T) \ \text{mod} \ \ell$ and $\overline{Q}(T):= Q(T) \ \text{mod} \ \ell$. This implies that the zeta function can be recovered as an application of the Chinese remainder theorem using the polynomials $\tilde{P}_{i}(T)$ for finitely many primes $\ell$. We now give bounds for the number and size for the primes required. Write $$\beta_{i}:=\dim \mathrm{H}^{i}(X, \mathbb{Q}_{\ell})= \deg P_{i}(X/\mathbb{F}_{q}, T)$$ for the $i^{\text{th}}$ $\ell$ -- adic Betti number of $X$. By \cite[\S 4.2]{rsv}, we know $\beta_{1}=\beta_{3}\leq 2D^{2}$ and $\beta_{2}\leq 2D^{N+1}$. As a result of Deligne's proof \cite{Weili} of the Weil-Riemann hypothesis for $X$, we know that the reciprocal roots of $P_{i}(X/ \mathbb{F}_{q}, T)$ have absolute value $q^{i/2}$. This implies that the coefficients of each polynomial $P_{i}(T)$ are bounded above by
$$
\binom{2D^{N+1}}{D^{N+1}}q^{D^{N+1}}.
$$

In particular, it suffices to compute $P_{i}(T) \ \text{mod} \ \ell$ for all primes $\ell \leq A \log q$ where $A=9\cdot D^{N+1}+3$. Further, observe that
$$
\frac{d}{dT}\log Z(X/\mathbb{F}_{q}, T)=\sum_{j=1}^{\infty}\#X(\mathbb{F}_{q^{j}})T^{j-1}=\frac{Q(T)\dot{P}(T)-P(T)\dot{Q}(T)}{P(T)Q(T)},
$$
so $\#X(\mathbb{F}_{q})$ can be read off as the constant term of the power-series expansion associated to the logarithmic derivative of $Z(X/ \mathbb{F}_{q}, T)$. 
\begin{remark}
    We note that we may need to work over field extensions $\mathbb{F}_{Q}/\mathbb{F}_{q}$ (e.g., to ensure the existence of a smooth fibre of $\pi$) and compute the $F_{Q}$ -- zeta function. The base zeta function can be recovered from any two such, via a recipe due to Kedlaya \cite[\S 8]{ked}.

\end{remark}

\section{Height bounds}
\label{app:height}
In this section, we recall the theory of heights and state certain height bounds to complement our algorithms. \\ \\
Let $K/\mathbb{Q}$ be a number field. Denote by $M_{K}$ the set of places of the ring of integers $\mathcal{O}_{K}$ and denote by $v_{\mathfrak{p}}$ for $\mathfrak{p}\in M_{K}$ the associated $\mathfrak{p}$ -- adic valuation. Let ${K}_{\mathfrak{p}}$ denote the completion of $K$ and set $n_{v_{\mathfrak{p}}}=[{K}_{\mathfrak{p}}: \mathbb{Q}_{\mathfrak{p}}]$.

\begin{definition}
    Let $P=[x_{0}:\ldots : x_{N}]\in \mathbb{P}^{N}({K})$ be a point. The \textit{Weil height} $h(P)$ is defined as
    $$
    h(P)=\frac{1}{[{K}:\mathbb{Q}]}\sum_{\mathfrak{p}}n_{v_{\mathfrak{p}}}\cdot\left(\log (\max_{j}\Vert x_{j} \Vert_{v_{\mathfrak{p}}})\right).
    $$
\end{definition}

\begin{definition}
Let $C$ be a curve over $K$ and let $J$ denote its Jacobian. The \textit{N\'eron-Tate height}, denoted $\hat{h}$ for a point $P\in J$ is defined as follows
\begin{equation}
    \hat{h}(P):=\lim_{j\rightarrow \infty}\frac{h(2^{j}P)}{4^{j}}.
\end{equation}

\end{definition}
It is clear that the N\'eron-Tate height vanishes on torsion points. We next recall the following, that relates the two height functions introduced above, on an abelian variety.

\begin{theorem}[Zarhin-Manin]
\label{thm:zarhin}

    Let $A$ be a polarised abelian variety over a number field $K$, together with an ample, symmetric line bundle $\Theta$.
    Then, there exist constants $c_{1}$ and $c_{2}$, depending on $A$ and $g$ such that for any $P\in A(\overline{{K}})$,
    \begin{equation}
    \label{eqn:zarhin}
    \hat{h}(P)- c_{1} \leq h(P)\leq \hat{h}(P)+c_{2}
    \end{equation}
    with
    $$
    c_{1}=\left(\frac{2^{2g-1}}{3}+1\right)\cdot h_{\Theta}(A)+\left(2^{2g-2}+\frac{67}{12}\right)\cdot g\cdot \log 2 \ \ \mathrm{and} \ \ 
    c_{2}=(2^{2g}-1)\cdot h_{\Theta}(A)+(2^{2g+1}-\frac{1}{3})\cdot g\cdot \log 2,
    $$

    where $h_{\Theta}(A)$ is the height of the neutral element $0_A$ of $A$.
\end{theorem}
\begin{proof}
  Apply \cite[3.2]{zarhinmanin} to the divisor $4\cdot \Theta$.
\end{proof}

\begin{theorem}[Height of torsion point]
\label{thm:tors_height}
    Let $C\subset \mathbb{P}^{N}$ be a smooth, projective curve of genus $g$ and degree $D$ over a number field ${K}$, and denote by $J$ its Jacobian. Let $\ell$ be a prime number, and let $P\in J[\ell]$ be an $\ell$ -- torsion point. Consider the embedding of $J$ into $\mathbb{P}^{M}$ given by Theorem~\ref{thm:eqnjac}. Then, we have
    $$
    \vert h(P)\vert \leq C, 
    $$
    where $C$ is a constant that depends only on $N$, $g$, $D$, the height of the coefficients of the equations defining $C$, the extension degree, and the logarithm of the discriminant of the number field ${K}/\mathbb{Q}$. The dependence is polynomial in the last three items. In particular, the height of an $\ell$ -- torsion point is bounded by a quantity independent of $\ell$.
\end{theorem}

\begin{proof}
    As $P$ is assumed to be torsion, we know $\hat{h}(P)=0$. We note firstly, that by Theorem~\ref{thm:edited}, the height of the Jacobian constructed in Theorem~\ref{thm:eqnjac} is bounded above by the height associated to the $4\cdot \Theta$ -- embedding. The result then follows from Theorem~\ref{thm:zarhin}, combined with the results of \cite[\S 2]{pazuki} and \cite[\S 1]{remond}. 
\end{proof}

\begin{remark}
     Theorem~\ref{thm:tors_height} holds with the base field $K$ replaced by a function field $\mathbb{F}_{q}(t)$ or a function field over a number field $K(t)$. We merely change the notion of height; in the former case, one uses a geometric height function, and in the latter case, a height function that captures both the geometric and arithmetic data, such as Moriwaki's height function \cite{moriwaki}. The general underlying principle is that the naive height only differs from the canonical height by a bounded amount (see \cite[\S 4]{silverman}).
\end{remark}
We now recall a result of Javanpeykar, which resolves a conjecture of Edixhoven-de Jong-Schepers \cite[Conjecture 5.1]{edcovers}, that bounds the Faltings height of the Jacobian of a ramified covering of the projective line.
\begin{theorem}[Javanpeykar]
\label{javanp}
   Let $U\subset \mathbb{P}^{1}_{\mathbb{Z}}$ be a nonempty open subscheme. There exist integers $a,b\in \mathbb{Z}_{>0}$ such that for any prime $\ell$, and any connected finite \'etale cover $$\Psi: V\rightarrow U_{\mathbb{Z}[1/\ell]},$$ the Faltings height of the Jacobian of the normalisation of $\mathbb{P}^{1}$ in the function field of $V$ is bounded by $$(\deg \Psi)^{a}$$ where $a$ is a constant that depends only on the height of $Z=\mathbb{P}^{1}_{\mathbb{Q}}\setminus U_{\mathbb{Q}}$ and the action of $\mathrm{Gal}(\overline{\mathbb{Q}}/\mathbb{Q})$ -- on $Z$. In particular, 
   $$
   a=6+\log\left(13\cdot 10^{6} \mathrm{A}\cdot (4 \mathrm{A} \mathrm{B})^{45\mathrm{A}^{3}2^{\mathrm{A}-2}\mathrm{A}!}\right)
    $$
    where $\mathrm{A}$ is the number of elements in the orbit of $Z$ under the action of $\mathrm{Gal}(\overline{\mathbb{Q}}/\mathbb{Q})$ and $\mathrm{B}$ is a bound for the height of the elements of $Z$.
    
\end{theorem}

\begin{proof}
   See \cite[Theorem 6.0.6]{javan}.
\end{proof}

\section{Results of Igusa}
\label{app:igusa}
In this appendix, we recall certain results of Igusa related to fibre systems of Jacobian varieties, their embeddings, and specialisation. This is then applied to the context of a Lefschetz pencil on a surface and the specialisation of the $\ell$ -- torsion in the Jacobian of the generic fibre. The treatment is based on the works \cite{igusa, igusaii, igusaabs}.

\hfill 

Let $\mathcal{X}\subset \mathbb{P}^{N}$ be a nice surface over a number field $K$ and let $\pi:\mathcal{X}\rightarrow \mathbb{P}^{1}$ be a Lefschetz pencil of hyperplane sections. Denote by $Z\subset \mathbb{P}^{1}$ the finite subset parametrising the nodal fibres and let $U=\mathbb{P}^{1}\setminus Z$. Let $\overline{\eta}\rightarrow \mathbb{P}^{1}$ be a geometric generic point and let the genus of the generic fibre $\mathcal{X}_{\overline{\eta}}$ (as a curve over the field $\overline{K}(t)$) be $g$. Write $\mathcal{F}:=R^{1}\pi_{\star}\mu_{\ell}$ for the derived pushforward. Consider an embedding of the Jacobian $\mathcal{J}_{\overline{\eta}}=\mathrm{Pic}^{0}(\mathcal{X}_{\overline{\eta}})$ into a projective space $\mathbb{P}^{M}$ \footnote{using e.g., Chow's method (\cite{chow} or \cite[Appendix]{igusa}) or Anderson's method (\cite{anderson}) sketched in Appendix~\ref{app:jac}, both of which involve the $\Theta$ -- divisor}. 

\begin{theorem}
\label{thm:igspec}
For $z\in \mathcal{Z}$, let $\widetilde{\mathcal{J}}_{z}$ be the specialisation of $\mathcal{J}_{\overline{\eta}}$ to $z$, over the specialisation $\mathcal{X}_{\overline{\eta}}\rightarrow \mathcal{X}_{z}$. Then, $\tilde{\mathcal{J}}_{z}$ is the completion of the generalised Jacobian \footnote{also called Rosenlicht variety} $\mathcal{J}_{z}$ of $\mathcal{X}_{z}$.
\end{theorem}
\begin{proof}
    See \cite[Theorem 3]{igusa}.
\end{proof}
\begin{theorem}
\label{thm:igsing}
    The singular locus of $\widetilde{\mathcal{J}}_{z}$ is $\widetilde{\mathcal{J}}_{z}\setminus \mathcal{J}_{z}$. Further, if $\omega$ is a $\overline{K}(t)$ -- rational point of $\mathcal{J}_{\overline{\eta}}$, then the specialisation $\omega_{z}$ of $\omega$ to $z$ is a smooth point of $\widetilde{\mathcal{J}}_{z}$.
\end{theorem}
\begin{proof}
    See \cite[pg 746, Theorem 1]{igusaii}.
\end{proof}
Now, under the natural inclusion $\overline{K}(t)\hookrightarrow \overline{K}((t-z))$, fix an embedding $\overline{K(t)}\hookrightarrow \overline{K}\langle \langle t-z \rangle \rangle$. As we saw in Section~\ref{subsec:cospec}, this completely determines a cospecialisation map $\phi_{z}:\mathcal{F}_{z}\hookrightarrow \mathcal{F}_{\overline{\eta}}$. We have the following.

\begin{theorem}
\label{thm:igtors}
    Write $\varsigma$ for the $0$ -- cycle on $\mathcal{J}_{\overline{\eta}}$ comprising of its $\ell$ -- torsion $\mathcal{J}_{\overline{\eta}}[\ell]$. Then the specialisation of $\varsigma$ to $z$ is the $0$ -- cycle on $\widetilde{\mathcal{J}}_{z}$ written $\overline{\varsigma}+\overline{\varsigma}'$ where $\overline{\varsigma}$ consists of the $\ell$ -- torsion of the generalised Jacobian $\mathcal{J}_{z}[\ell]$ and $\overline{\varsigma}'$ is a positive cycle, each of which is a multiple point of $\widetilde{\mathcal{J}}_{z}$ arising from the singularities of the curve $\prescript{(\ell)}{}{}\mathfrak{C}\subset \mathbb{P}^{M}$ over $\overline{K}$ corresponding to the $\ell$ -- division ideal $\prescript{(\ell)}{}{}\mathcal{I}_{\overline{\eta}}$ of $\mathcal{J}_{\overline{\eta}}$.
\end{theorem}
\begin{proof}
    See \cite[Theorem 2]{igusaii}.
\end{proof}
\begin{theorem}
\label{thm:igg}
    Let $\gamma\in \mathcal{F}_{\overline{\eta}}\setminus \phi_{z}(\mathcal{F}_{z})$. Then $\sigma_{z}(\gamma)$ and $\gamma$ specialise to the same point in $\widetilde{\mathcal{J}}_{z}$. Further, $\sigma_{z}(\gamma)-\gamma$ lies in the space generated by the vanishing cycle at $z$.
\end{theorem}
\begin{proof}
    See the proof of \cite[Theorem 3]{igusaii}.
\end{proof}

\begin{theorem}
\label{thm:igred}
    Now, consider $\mathcal{J}_{\overline{\eta}}$ as being defined over $\overline{K}((t-z))$. Then, all the points of $\phi_{z}(\mathcal{F}_{z})$ are rational over $\overline{K}((t-z))$ and the splitting field $\mathbb{K}$ of $\mathcal{F}_{\overline{\eta}}$ over $\overline{K}((t-z))$ satisfies
    $$
    [\mathbb{K}: \overline{K}((t-z))]=\ell,
    $$
    i.e., $\mathbb{K}$ is the field obtained by adjoining $\overline{K}((t-z))$ with an $\ell^{\mathrm{th}}$ -- root of $t-z$.
\end{theorem}

\begin{proof}
    See \cite[Theorem 2]{igusaabs}.
\end{proof}

\section{Abstract Abel map and embeddings of Jacobians}
\label{app:jac}
 This section of the appendix aims to provide equations for the Jacobian of smooth projective curves and the generalised Jacobian of a nodal curve. A construction of the Jacobian of a smooth curve was described by Chow \cite{chow}; however, our treatment follows Anderson \cite{anderson}, who provides an `elementary' algebraic construction of the Abel map \cite{andersonabel}. In \cite{andersontheta}, it is shown that the construction matches with an `edited' $4\cdot \Theta$ -- embedding associated to the $\Theta$ -- divisor on the Jacobian of a curve. \\ 
We explain briefly Anderson's construction of the `abstract Abel map'. Let $C\subset \mathbb{P}^{N}$ be a smooth, projective curve of genus $g$ over a field $\mathbb{K}$. Let $\mathcal{E}$ be a line bundle of degree $w\geq 2g+1$ and let $\mathcal{D}$ be a line bundle of degree zero. Let $\underline{u}$ be a basis for $\mathrm{H}^{0}(C, \mathcal{D}^{-1}\otimes \mathcal{E})$ and let $\underline{v}$ be a basis for $\mathrm{H}^{0}(C, \mathcal{D}\otimes \mathcal{E})$. Denote by $C^{\{0,\ldots , w+1\}}$ the $w+2$ -- fold power of $C$ with numbering remembered, and for a section $f$ of a line bundle on $C$, denote by $f^{(i)}$ the pullback by the $i^{\mathrm{th}}$ projection. Then the abstract Abel map sends $\mathcal{D}$ to the $w\times w$ matrix with entries
\begin{equation}
\label{eqn:abelmap}
\mathrm{abel}(\mathcal{D})_{ij}= \begin{vmatrix} 
    \widehat{\underline{v}^{(0)}} \\
    \vdots  \\
    \widehat{\underline{v}^{(i)}} \\
    \vdots 
\end{vmatrix}
    \cdot
    \begin{vmatrix} 
    \vdots  \\
    \widehat{\underline{u}^{(i)}} \\
    \vdots \\
    \widehat{\underline{u}^{(w+1)}}
    \end{vmatrix}
\cdot 
\begin{vmatrix} 
    \vdots  \\
    \widehat{\underline{v}^{(j)}} \\
    \vdots \\
    \widehat{\underline{v}^{(w+1)}}
     \end{vmatrix}
    \cdot
    \begin{vmatrix} 
    \widehat{\underline{u}^{(0)}} \\
    \vdots  \\
    \widehat{\underline{u}^{(j)}} \\
    \vdots 
\end{vmatrix}
\end{equation}
for $1\leq i, j\leq w$, where the leftmost term in the product denotes the determinant of the $w\times w$ matrix obtained by stacking the $\underline{v}^{(t)}$ as row vectors numbered $0$ to $w+1$ and removing the rows numbered $0$ and $i$. In particular, the construction maps classes of degree zero line bundles to $w\times w$ matrices with the entry from the $i^{\mathrm{th}}$ row and $j^{\mathrm{th}}$ column being from the space
$$
\mathrm{H}^{0}\left(C^{\{0, \ldots , w+1\}} , \frac{\bigotimes_{s=0}^{w+1}\left(\mathcal{E}^{(s)}\right)^{\otimes 4}}{\left(\mathcal{E}^{0}\right)^{\otimes 2}\otimes \left(\mathcal{E}^{(i)}\right)^{\otimes 2}\otimes \left(\mathcal{E}^{(j)}\right)^{\otimes 2}\otimes \left(\mathcal{E}^{(w+1)}\right)^{\otimes 2}}\right).
$$
In summary, the abstract Abel map gives a way to realise any degree zero divisor on $C$ as a point on its Jacobian, embedded into projective space.
\\ \\ We now sketch below how to obtain the equations for the Jacobian, i.e., the ideal of polynomials vanishing on the image of the abstract Abel map.
 
 \begin{itemize}
     \item[(1)] Fix an effective divisor $E$ of $C$ with $\deg(E)\geq 2g+1$.
     \item[(2)] Set $w=\dim \mathcal{L}(E)=\deg(E)-g+1$.
     \item[(3)] Write $S=\mathrm{supp}(E)$, $A=\mathrm{H}^{0}(S, \mathcal{O}_{C})$ and $L=\mathcal{L}(2E)$.
 \end{itemize}
Then, the Jacobian of $C$ is given by the projective algebraic variety $J$ of $\mathbb{K}$ -- proportionality classes of Jacobi matrices of type $(\mathbb{K}, w, A, L)$. A proof is given in \cite[Theorem 4.4.6]{anderson}. From \cite[3.7.3]{anderson}, we see that the complexity of the construction is at worst $\exp(\mathrm{poly}(g))$. \\ \\
In the case $\mathbb{K}=k(t)$ is the function field of the projective line, and $C$ is a curve over $\mathbb{K}$, we want to choose an effective divisor $E$ on $C$ for the embedding so that upon specialisation to a smooth value $t=u$, the corresponding embedding of the Jacobian of $C_u$ is given by $E_u$. This is achieved as follows.
\begin{itemize}
    \item Choose an effective divisor $E$ of $C$ of degree $\geq 2g+1$ via taking all the zeros of a rational function $\lambda$ on $C$, with $k(t)$ -- coefficients. We may assume $\mathrm{div}(\lambda)=\lambda_{+} - \lambda_{-}$, with $\lambda_+$ and $\lambda_-$ effective of degree $\geq 2g+1$, and no redundancies between them. Also assume that the divisor $\mathcal{E}$ specialised to any $u\in \mathbb{P}^{1}$ contains no singular point of $\mathcal{X}_{u}$ in its support.

\item For a smooth point $u$, the associated divisor $E_u$ is obtained by specialising $\lambda_+$ to $u$. 

\item The Jacobian of the curve $C_u$ corresponds to the specialisation of the Jacobian of $C$ at $t=u$, via the divisor $E_u$. 
    
\end{itemize}

 \begin{algorithm}
      \caption{\texttt{Abstract Abel map and its inverse on} $\ell$ -- \texttt{torsion} }
     \label{algo:abel}
\begin{itemize}
    \item \textbf{Input:} The generic fibre $\mathcal{X}_{\overline{\eta}}$ of a Lefschetz pencil $\pi:\mathcal{X}\rightarrow \mathbb{P}^{1}$ on a smooth projective surface $\mathcal{X}$ over a number field $K$, and a degree zero divisor $D\in \mathrm{Pic}^{0}(\mathcal{X}_{\overline{\eta}})[\ell]$ represented using Theorem~\ref{thm:huaier}.

\item \textbf{Output:} The image $\mathrm{abel}(D)$ of the map in (\ref{eqn:abelmap}) as a point in projective space $\mathbb{P}^{M}$ lying on the Jacobian $\mathcal{J}_{\overline{\eta}}$, satisfying the conditions of the paragraph above.
    
\end{itemize}
     \begin{algorithmic}[1]
     \STATE  Choose an effective divisor ${E}$ of $\mathcal{X}_{\overline{\eta}}$ of degree $w\geq 2g+1$ via taking all the zeros of a rational function $\lambda$, with $K(t)$ -- coefficients on $\mathcal{X}_{\overline{\eta}}$. We may assume $\mathrm{div}(\lambda)=\lambda_{+} - \lambda_{-}$, with $\lambda_+$ and $\lambda_-$ effective of degree $\geq 2g+1$, and no redundancies between them. Also assume that the divisor ${E}$ specialised to any $u\in \mathbb{P}^{1}$ contains no singular point of $\mathcal{X}_{u}$ in its support.

     \STATE Compute bases $\underline{v}$ for $\mathrm{H}^{0}(\mathcal{X}_{\overline{\eta}}, E+D)$ and $\underline{u}$ for $\mathrm{H}^{0}(\mathcal{X}_{\overline{\eta}}, E-D)$ using an effective Riemann-Roch algorithm via Theorem~\ref{thm:jacarith}.

     \STATE Maintaining $w+2$ sets of variables, compute the pullbacks $\underline{u}^{(i)}$ and $\underline{v}^{(j)}$ for each $i,j\in \{0, \ldots, w+1\}$. These are merely the same rational functions associated to a specific set of variables.

     \STATE Compute the map (\ref{eqn:abelmap}) using these pullbacks.

      \STATE For any $u\in \mathbb{P}^{1}$, the embedding of the Jacobian $\mathrm{Pic}^{0}(\mathcal{X}_{u})\hookrightarrow \mathbb{P}^{M}$ is given by the divisor $E_u$. If we specialise the input divisor $D$ to $u$, we get $D_u\in \mathrm{Pic}^{0}(\mathcal{X}_{u})[\ell]$. 
      
      \STATE To invert the Abel map on $\mathrm{Pic}^{0}(\mathcal{X}_{u})[\ell]$, given a point in $\mathbb{P}^{M}$ corresponding to an element of $\mathrm{Pic}^{0}(\mathcal{X}_{u})[\ell]$, we simply go through all the $\ell^{2g}$ divisorial representatives of $\ell$ -- torsion as a result of the algorithm from Theorem~\ref{thm:huaier} and check which of them map to our given point via the divisor $E_u$ and the map (\ref{eqn:abelmap}). There will be a unique pre-image as the Abel map is injective.

     \end{algorithmic}
 \end{algorithm}

\begin{remark}
    The only dependence on $\ell$ in Algorithm~\ref{algo:abel} is the input divisor $D\in \mathrm{Pic}^{0}(\mathcal{X}_{\overline{\eta}})[\ell]$. By Theorem~\ref{thm:huaier}, we know that $D$ can be efficiently represented $\mathrm{poly}(\ell)$ time and the bases for the Riemann-Roch spaces $\mathrm{H}^{0}(\mathcal{X}_{\overline{\eta}}, E\pm D)$ are computed using Theorem~\ref{thm:jacarith}.
\end{remark}

 By \cite[Theorem 3]{igusa} (see also \cite{igusaii}), we know that the specialisation of the Jacobian of the generic fibre $\mathcal{X}_{\overline{\eta}}$ of a Lefschetz pencil $\pi:\mathcal{X}\rightarrow \mathbb{P}^{1}$ on a surface $\mathcal{X}$ to a singular $z\in \mathcal{Z}$ is the completion of the generalised Jacobian of $\mathcal{X}_{z}$. In summary, we have the following.
\begin{theorem}
\label{thm:eqnjac}
Let $\mathcal{X}\subset \mathbb{P}^{N}$ be a nice surface of degree $D$ over a number field $K$ and let $\pi:\mathcal{X}\rightarrow \mathbb{P}^{1}$ be a Lefschetz pencil of hyperplane sections on $\mathcal{X}$. Let $U\subset \mathbb{P}^{1}$ be the subscheme parametrising the smooth fibres and let $Z=\mathbb{P}^{1}\setminus U$ parametrise the singular nodal fibres. Then, there exists an algorithm that computes 
\begin{itemize}
    \item [(i)] the Jacobian $\mathcal{J}_{\overline{\eta}}$ of $\mathcal{X}_{\overline{\eta}}$ in a projective space $\mathbb{P}^{M}$ as a system of homogeneous polynomial equations,

    \item[(ii)] an explicitisation of the Abel map $\mathcal{X}_{\overline{\eta}}\hookrightarrow \mathcal{J}_{\overline{\eta}}$,

    \item[(iii)] an explicit addition law on the Jacobian $\mathcal{J}_{\overline{\eta}}$ with atlases, in the sense of Pila \cite{pila}. This provides a translation between the language of divisor arithmetic on $\mathcal{X}_{\overline{\eta}}$ and points on $\mathcal{J}_{\overline{\eta}}$. Moreover, for any specialisation to $u\in \mathbb{P}^{1}$, the group law on $\mathcal{J}_{\overline{\eta}}$ specialises to that on $\mathcal{J}_{u}$.

    
\end{itemize}
\end{theorem}
\begin{proof}
    See \cite[\S 4]{anderson}.
\end{proof}
\begin{theorem}
\label{thm:edited}
    The embedding described in Theorem~\ref{thm:eqnjac} factors through (and corresponds exactly to, upto linear hull) an `edited' $4\cdot \Theta$ -- embedding, i.e., the complete linear system associated to the divisor $4\cdot \Theta$ on the Jacobian, consisting of those theta-functions which vanish at the origin with order $\leq 1$.
\end{theorem}
\begin{proof}
    See \cite[\S 3]{andersontheta}.
\end{proof}


        

    

    



    

\end{document}